\documentclass[12pt,british]{article}

\usepackage{babel}
\usepackage{amsmath,amsthm}
\usepackage{amssymb}
\usepackage[utf8]{inputenc}
\usepackage{multicol}
\usepackage[all]{xy}
\usepackage{bbm,enumerate}
\usepackage{geometry}
\newtheorem{theorem}{Theorem}

\newtheorem{lemma}[theorem]{Lemma}
\newtheorem{proposition}[theorem]{Proposition}

\theoremstyle{definition}

\newtheorem{remark}[theorem]{Remark}

\newcommand{\torus}{\mathbb{T}^2}
\newcommand{\klein}{\mathbb{K}^2}
\newcommand{\z}{\mathbb{Z}}
\newcommand{\zsz}{\mathbb{Z} \oplus \mathbb{Z}}
\newcommand{\zsdz}{\mathbb{Z} \rtimes \mathbb{Z}}
\newcommand{\ztwo}{\mathbb{Z}_2}
\newcommand{\rtwo}{\mathbb{R}^2}
\newcommand{\id}{\mathbbm{1}}
\newcommand{\free}{\mathcal{F}}
\renewcommand{\hom}{{\rm Hom}}

\begin{document}

\title{The Borsuk-Ulam property for homotopy classes of \\ selfmaps of surfaces of Euler characteristic zero}

\author{DACIBERG LIMA GON\c{C}ALVES
~\footnote{Departamento de Matemática, IME, Universidade de São Paulo, Caixa Postal 66281, Ag.\ Cidade de São Paulo, CEP: 05314-970, São Paulo, SP, Brazil. 
e-mail: \texttt{dlgoncal@ime.usp.br}}
\and
JOHN GUASCHI
~\footnote{Normandie Univ., UNICAEN, CNRS, Laboratoire de Mathématiques Nicolas Oresme UMR CNRS~\textup{6139}, 14000 Caen, France.
e-mail: \texttt{john.guaschi@unicaen.fr}}
\and
VINICIUS CASTELUBER LAASS
~\footnote{Departamento de Matemática, IME, Universidade Federal da Bahia, Av.\ Adhemar de Barros, S/N Ondina CEP: 40170-110, Salvador, BA, Brazil. 
e-mail: \texttt{vinicius.laass@ufba.br}} 
}
\date{29th July 2016}
\maketitle


		                \begin{abstract}
Let $M$ and $N$ be topological spaces  such that $M$ admits a free involution $\tau$.  A homotopy class $\beta \in [ M , N ] $ is said to have the {\it Borsuk-Ulam property with respect to $\tau$} if for every representative map $f: M \to N$ of $\beta$, there exists a point $x \in M$ such that $f ( \tau ( x) ) = f(x)$. In the case where $M$ is a compact, connected manifold without boundary and $N$ is a compact, connected surface without boundary different from the $2$-sphere and the real projective plane, we formulate this property in terms of the pure and full $2$-string braid groups of $N$, and of the fundamental groups of $M$ and the orbit space of $M$ with respect to the action of $\tau$. 
If $M=N$ is either the $2$-torus $\torus$ or the Klein bottle $\klein$, we then solve the problem of deciding which homotopy classes of $[M,M]$ have the Borsuk-Ulam property. First, if $\tau : \torus \to \torus$ is a free involution that preserves orientation, we show that no homotopy class of $[ \torus , \torus ]$ has the Borsuk-Ulam property with respect to $\tau$. Secondly, we prove that up to a certain equivalence relation, there is only one 
class of free involutions $\tau : \torus \to \torus$ that reverse orientation, and for such involutions, we classify the homotopy classes
in $[\torus, \torus]$ that have the Borsuk-Ulam property with respect to $\tau$ in terms of the induced homomorphism on the fundamental group. Finally, we show that if $\tau: \klein \to \klein$ is a free involution, then a homotopy class of $[\klein , \klein ]$ has the Borsuk-Ulam property with respect to $\tau$ if and only if the given homotopy class lifts to the torus.
                \end{abstract}
                
\begingroup
\renewcommand{\thefootnote}{}
\endgroup

\maketitle

\noindent

                \section{Introduction}\label{introduction}

In the early twentieth century, St.~Ulam conjectured that if $ f : \mathbb{S}^n \to \mathbb{R}^n $ is a continuous map, there exists $ x \in \mathbb{S}^n$ such that $ f ( A (x)) = f (x) $, where $ A: \mathbb{S}^n \to \mathbb{S}^n $ is the antipodal map. The confirmation of this result by K.~Borsuk in 1933~\cite{Borsuk}, known as the Borsuk-Ulam theorem, was the beginning of what we now refer to as {\it Borsuk-Ulam type theorems} or the {\it Borsuk-Ulam property}. More information about the history and some applications of the Borsuk-Ulam theorem may be found in~\cite{Mato}, for example.

One possible generalisation of the classical Borsuk-Ulam theorem is to substitute $\mathbb{S}^n$ and $\mathbb{R}^n$ by other spaces and to replace the antipodal map by a free involution. A natural question is the following: does every continuous map collapse an orbit of the involution? More precisely, given topological spaces $M$ and $N$  such that $M$ admits a free involution $\tau$, we say that the triple {\it $(M , \tau ; N)$ has the Borsuk-Ulam property} if for every continuous map $f: M \to N$, there exists a point $x \in M$ such that $f(\tau(x)) = f(x)$. Note that in the whole of this paper, we suppose without further comment that our topological spaces are connected. 

In~\cite{Gon}, if $M$ is a compact surface without boundary, D.~Gon\c{c}alves presented a complete description of the triples $(M, \tau; \mathbb{R}^2 )$ that have the Borsuk-Ulam property. In~\cite{GonGua}, if $M$ and $N$ are compact surfaces without boundary, D.~Gon\c{c}alves and J.~Guaschi described the triples $(M, \tau; N)$ that have the Borsuk-Ulam property. If the triple $(M,\tau; \rtwo)$ does not have the Borsuk-Ulam property, by definition, there exists a continuous map $f: M \rightarrow \rtwo$ such that $f( \tau (x)) \neq f(x)$ for all $x \in M$. The fact that there is a single homotopy class of maps from $M$ to $\rtwo$ implies that if $g: M \to \rtwo$ is a continuous map such that $g( \tau (x)) = g(x)$ for some $x \in M$, then $g$ is homotopic to $f$. In other words, $g$ is homotopic to a continuous map that does not collapse the orbits of the involution $\tau$.

The situation is different if the triple $(M , \tau ; N)$ does not have the Borsuk-Ulam property and the set $[M,N]$ has cardinality greater than one. Once more, there exists a map $f: M \to N$ such that $f(\tau(x)) \neq f(x)$ for all $x \in M$. But if $g: M \to N$ is a continuous map such that $g(\tau(x)) = g(x)$ for some $x \in M$, we do not know whether $g$ is homotopic to a map that is injective on each orbit of the involution, unless $g$ is homotopic to $f$.   From these observations, we have a natural refinement of the Borsuk-Ulam property, in the following way: we say that \emph{a homotopy class $\beta\in [M, N]$ satisfies the Borsuk-Ulam property with respect to $\tau$} if for every map $f:M \to N$, where $f\in \beta$, there exists $x\in M$ such that $f(\tau(x))=f(x)$. In conjunction with~\cite{GonGua}, these observations give rise to the following Borsuk-Ulam problem: {\it given compact surfaces $M$ and $N$ without boundary and a free involution $\tau: M \to M$, classify the elements of the set of homotopy classes $\left[ M , N \right]$ that have the Borsuk-Ulam property.} Before Proposition~\ref{reduce_involution}, we recall the definition of  an equivalence relation on free involutions. This relation is suitable for  the study of this problem. 

In this paper, we solve this problem for the cases  where $M$ and $N$ coincide and are compact surfaces without boundary of Euler characteristic zero, namely the $2$-torus $\torus$ or the Klein bottle $\klein$. We have three main results. First, suppose that $M=\torus$.

\begin{theorem}\label{BORSUK_TAU_1} Let $\tau: \torus \to \torus$ be a free involution that preserves orientation.
If $\beta \in [\torus , \torus]$ is a homotopy class then $\beta$ does not have the Borsuk-Ulam property with respect to $\tau$.
\end{theorem}

Let us consider free involutions of the torus that reverse orientation. We will show that there is only one equivalence class of such involutions. Let $\tau_2: \torus \to \torus$ be the  orientation-reversing involution that  admits the  lifting to the plane given by $\widehat{\tau}_{2}(x,y )= ( x+\frac{1}{2},1-y )$ for all $(x,y)\in \rtwo$ (see  Section~\ref{borsuk_maps_torus}).

\begin{theorem}\label{BORSUK_TAU_2}
Let $\beta \in [\torus, \torus]$ be a homotopy class and let
    $\begin{pmatrix}
          \beta_{1,1} & \beta_{1,2} \\
          \beta_{2,1} & \beta_{2,2}
    \end{pmatrix}$
be the integral matrix of the homomorphism induced by $\beta$ on the fundamental group. Then $\beta$ has the Borsuk-Ulam property with respect to $\tau_2$ if and only if $( \beta_{1,1} , \beta_{2,1}) \neq ( 0 , 0)$, and $\beta_{1,2}$ and $\beta_{2,2}$ are both even.
\end{theorem}

If $\tau_2'$ is a free involution that reverses orientation it follows from Proposition~\ref{class_involutions_torus}   that $\tau_2$ and $\tau_2'$ are equivalent. In conjunction with Proposition~\ref{reduce_involution} and  Theorem~\ref{BORSUK_TAU_2} this enables us to  classify the homotopy classes $\beta$ that satisfy the Borsuk-Ulam property with respect to any such free involution $\tau_2'$.
 
Now suppose that the surface under consideration is the Klein bottle $\klein$. 
\begin{theorem}\label{BORSUK_TAU_3}  Let $\tau: \klein \to \klein$ be a free involution.
A homotopy class $\beta \in [\klein,\klein]$ has the Borsuk-Ulam property with respect to $\tau$ if and only if $\beta$ lifts to the torus.
\end{theorem}

This papers contains five sections besides the introduction. In Section~\ref{generalities}, we first provide an algebraic description of the  sets $[ M , N ]$ and $[M , m_1 ; N , n_1]$  (the set of pointed homotopy classes) in the case where $M$ and $N$ are manifolds without boundary such that $N$ is a $K ( \pi , 1 )$ (see Theorem~\ref{set_homotopy}). In Lemma~\ref{equiv_maps}, we determine an algebraic criterion to decide whether a given homotopy class contains a representative that is an equivariant map. If $N$ is a compact surface without boundary different from $\mathbb{S}^2$ and $\mathbb{RP}^2$, in Propositions~\ref{borsuk_braid} and~\ref{borsuk_pointed_free} we give an algebraic condition to decide whether a homotopy class of maps between $M$ and $N$ has the Borsuk-Ulam property, and in Proposition~\ref{invo_seq_exact} we show that the set of equivalence classes of free involutions of a compact surface without boundary different from $\mathbb{S}^2$ and $\mathbb{RP}^2$ is in one-to-one correspondence with the equivalence classes of certain short exact sequences. In Section~\ref{braid_torus} (resp.\ Section~\ref{braid_klein}), we study the $2$-string braid groups of $M$,  where $M=\torus$ (resp.\ $M=\klein$). In Theorem~\ref{presentation_p2_t2} (resp. Theorem \ref{presentation_p2_klein}), we give a presentation of $P_{2}(M)$, from which we deduce in Theorem~\ref{remark_p2_torus} (resp.\ Theorem~\ref{braid_group_klein})  that $P_{2}(M)$ is isomorphic to $F( x , y ) \oplus \z \oplus \z$~(\emph{cf.}~\cite[Lemma 17]{BelGerGua}), where $F( x , y )$ is the free group on the set $\{x, y\}$ (resp.\ to a semi-direct product of the form $F( u , v ) \rtimes_\theta ( \zsdz )$). In Theorem~\ref{remark_p2_torus} (resp.\ Theorem~\ref{action_sigma_k2}), we describe the action by conjugation of the element $\sigma \in B_2 (M) \backslash P_2(M)$ on $P_2 (M)$. In Section~\ref{borsuk_maps_torus}, in Theorem~\ref{class_involutions_torus}, we show that up to the above-mentioned equivalence relation, there are precisely  two classes of free involutions of the torus that correspond to the orientation-preserving and orientation-reversing involutions respectively, and we develop some results and arguments necessary to prove Theorem~\ref{BORSUK_TAU_1} and~\ref{BORSUK_TAU_2}. Finally, in Section~\ref{borsuk_maps_klein}, we show in Theorem~\ref{class_involutions_klein} that up to the above-mentioned equivalence relation, there is just one class of free involutions of the Klein bottle, and we prove Theorem~\ref{BORSUK_TAU_3}.

\section{Preliminaries and Generalities}\label{generalities}
		
Let $(M,m_1)$ and $(N,n_1)$ be pointed manifolds, and suppose that $\pi_i ( N , n_1)$ is trivial for all $i \geq 2$. Let $[M, N]$ (resp.\ $[ M , m_1 ; N , n_1 ]$) denote the set of free (resp.\ pointed) homotopy class of maps from $M$ to $N$ (resp.\ from $(M,m_1)$ to $(N,n_1)$), and let $\hom ( \pi_1 ( M , m_1 ) , \pi_1 ( N , n_1 ))$ denote  the set of homomorphisms between the fundamental groups of $M$ and $N$. If $f,g: (M , m_1) \rightarrow (N , n_1)$ are homotopic pointed maps, denoted  by  $f \simeq g$ (rel.\ $m_1$), then the induced homomorphisms $f_\# , g_\#: \pi_1 ( M , m_1 ) \rightarrow \pi_1 ( N , n_1)$ on the level of fundamental groups are equal. Given a pointed homotopy class $\alpha \in [ M , m_1 ; N , n_1]$, we may thus associate a homomorphism $\alpha_\# \in \hom (  \pi_1 ( M , m_1) , \pi_1 ( N , n_1))$ by choosing a representative map $f: ( M , m_1) \to ( N , n_1)$ of $\alpha$, and by taking the induced homomorphism. This gives rise to the following well-defined map:
\begin{equation}\label{defGammaMN}
	\begin{array}{rccl}
	\Gamma_{M,N}: & [M, m_1 ; N , n_1] & \longrightarrow & \hom ( \pi_1 (M , m_1), \pi_1 (N , n_1)) \\
		     & \alpha = [f] & \longmapsto & \alpha_\# : = f_\#.
	\end{array}
\end{equation}
It is well known that the map $\Gamma_{M,N}$ is a bijection~\cite[Chapter V, Theorem 4.3]{White}.
If $f,g: (M , m_1) \rightarrow (N, n_1)$ are pointed maps such that $f \simeq g$ (rel. $m_1$), then omitting the base points, the maps $f,g: M \to N$ are homotopic. Given a pointed homotopy class $\alpha \in [ M , m_1 ; N , n_1]$, we may thus associate a free homotopy class $\alpha_\free\in [ M , N ]$ by choosing a representative map $f: ( M , m_1 ) \to ( N, n_1)$ of $\alpha$, and by taking the free homotopy class $\alpha_\free$ that represents the map $f : M \to N$, from which we obtain the following well-defined map:
\begin{equation}\label{defLambdaMN}
	\begin{array}{rccl}
		\Lambda_{M,N}: & [ M , m_1 ; N , n_1 ] & \longrightarrow & [ M , N ] \\
		& \alpha = [f] & \longmapsto & \alpha_\free : = [f],
	\end{array}
\end{equation}
that is surjective by~\cite[Lemma 6.4]{Vick}.
 
Two homomorphisms $h_1, h_2 \in \hom ( \pi_1 (M , m_1 ), \pi_1 (N , n_1 ))$ are said to be \emph{equivalent}, written $h_1 \sim h_2$, if there exists $\omega \in \pi_1 (N , n_1)$ such that $h_1 ( \upsilon ) = \omega h_2 ( \upsilon) \omega^{-1}$ for all $\upsilon \in \pi_1 (M,m_1)$. It is straightforward to see that $\sim$ is an equivalence relation. The associated canonical projection shall be denoted as follows:
\begin{equation}\label{defUpsilonMN}
\begin{array}{rccl}
\Upsilon_{M,N}: & \hom ( \pi_1 (M,m_1), \pi_1 (N,n_1)) & \longrightarrow & \dfrac{\hom ( \pi_1 (M , m_1), \pi_1 (N , n_1))}{\sim} .
\end{array}
\end{equation}
By~\cite[Chapter~V, Corollary 4.4]{White}, there exists a bijective map:
$$
\begin{array}{rccl}
\Delta_{M,N}: & [M,N] & \longrightarrow & \dfrac{\hom ( \pi_1 (M,m_1), \pi_1 (Y,n_1))}{\sim}	
\end{array}
$$
such that $\Delta_{M,N} \circ \Lambda_{M,N} =  \Upsilon_{M,N} \circ \Gamma_{M,N}$. Given an element $\beta \in [ M , N ]$, we denote the equivalence  class $\Delta_{M,N}( \beta)$ by $\beta_\#$. We sum up these observations in the following theorem, that we shall often use in this paper.

\begin{theorem}\label{set_homotopy}
If $(M,m_1)$ and $(N,n_1)$ are pointed manifolds such that $\pi_i ( N , n_1)$ is trivial for all $i \geq 2$, then the following diagram is commutative:
\begin{equation}\label{diag_homotopy}\begin{gathered}\xymatrix{
[M , m_1 ; N , n_1 ] \ar[r]^-{\Gamma_{M,N}} \ar[d]_-{\Lambda_{M,N}} & \hom ( \pi_1 (M,m_1), \pi_1 (N,n_1)) \ar[d]^-{\Upsilon_{M,N}} \\
[M , N ] \ar[r]^-{\Delta_{M,N}} & \dfrac{\hom ( \pi_1 (M,m_1), \pi_1 (N,n_1))}{\sim},}\end{gathered}\end{equation}
where the horizontal arrows are bijections, and the vertical arrows are surjective.		
\end{theorem}

Let $M$ be a manifold, and let $\tau: M \to M$ be a free involution. Let $M_\tau$ denote the corresponding \emph{orbit space}, which is also a manifold, and let  $p_\tau : M \to M_\tau$ denote the associated double covering. This gives rise to the following short exact sequence:
\begin{equation}\label{seq_tau}\xymatrix{
\quad 1 \ar[r] & \pi_1 ( M , m_1) \ar[r]^-{(p_\tau)_\#} 
& \pi_1 ( M_\tau , p_\tau(m_1)) \ar[r]^-{\theta_\tau} 
&  \ztwo \ar[r] & 1,}
\end{equation}
that we call the \emph{short exact sequence induced by $\tau$}, and that we denote by $S_\tau$, where we identify $\dfrac{\pi_1 ( M_\tau , p_\tau(m_1))}{( p_\tau )_\# ( \pi_1 ( M , m_1) )}$ with $\ztwo = \{ \overline{0} , \overline{1} \}$, and $\theta_{\tau}$ is  the natural projection onto the quotient.

We now prove an algebraic criterion to decide whether a pointed homotopy class has an equivariant representative map, which will help in simplifying the proofs of Propositions~\ref{borsuk_braid} and~\ref{invo_seq_exact}.

\begin{lemma}\label{equiv_maps}
Let $(M,m_1)$ and $(N,n_1)$ be pointed manifolds (resp.\ compact surfaces without boundary) such that $\pi_i ( N , n_1 )$ is trivial for all $i \geq 2$, let $\tau: M \to M$ and $\tau_1 : N \to N$ be free involutions. Given a pointed homotopy class $\alpha \in [M , m_1 ; N , n_1]$, the following conditions are equivalent:
\begin{enumerate}
\item there exists a representative map (resp.\ homeomorphism) $f: ( M , m_1 ) \to ( N , n_1)$ of $\alpha$ that is $(\tau, \tau_1)$-equivariant, i.e.\  $f( \tau(x)) = \tau_1 ( f(x))$ for all $x \in M$.
	
\item there exists a homomorphism (resp.\ isomorphism) $\psi: \pi_1 ( M_\tau , p_\tau (m_1) ) \to \pi_1 ( N_{\tau_1} , p_{\tau_1 } (n_1) )$ such that the following diagram is commutative:
\begin{equation}\label{diag_equiv}\begin{gathered}\xymatrix{%
\pi_1(M,m_1) \ar[rr]^-{ \alpha_\# } \ar[d]_-{ (p_\tau)_\# }  & & \pi_1 ( N , n_1) \ar[d]^-{ ( p_{\tau_1} )_\# } \\
\pi_1( M_\tau, p_\tau( m_1 )) \ar@{.>}[rr]^-{\psi} \ar[rd]_-{\theta_\tau} &  & \pi_1 ( N_{\tau_1}, p_{\tau_1} (n_1) ) \ar[ld]^-{ \theta_{\tau_1}}  \\
& \ztwo .&}\end{gathered}
\end{equation}
\end{enumerate}
\end{lemma}

\begin{proof}
$(1 \Rightarrow 2)$ Suppose first that $f: ( M , m_1 ) \to ( N , n_1)$ is a representative map of $\alpha$ that is
$(\tau, \tau_1)$-equivariant. Then $f \{ x , \tau(x) \} = \{ f(x) , \tau_1 ( f(x)) \}$ for all $x \in M$, and hence the map $f$ induces a map of the corresponding orbit spaces, in other words, there exists a map $\overline{f}: ( M_\tau , p_\tau (m_1)) \to ( N_{\tau_1} , p_{ \tau_1} (n_1))$ such that the following diagram is commutative:
\begin{equation}\label{diag_equiv_aux}\begin{gathered}\xymatrix{
(M , m_1) \ar[r]^-{f} \ar[d]_-{ p_\tau } & (N , n_1) \ar[d]^-{ p_{\tau_1} } \\
(M_\tau, p_\tau(m_1)) \ar[r]^-{ \overline{f} } & ( N_{\tau_1}, p_{\tau_1} (n_1)).}\end{gathered}
\end{equation}
Diagram~(\ref{diag_equiv_aux})  implies that $(p_{\tau_1})_\# \circ \alpha_\# = \psi \circ (p_\tau)_\#$, and also implies that $\theta_{\tau_1} \circ \psi = \theta_\tau$ using standard covering space arguments and the fact that the map $f$ is equivariant. We thus obtain the commutative diagram~(\ref{diag_equiv}). Further, if $f$ is a homeomorphism, then $\overline{f}$ is too, and so the induced homomorphism $\psi = \overline{f}_\# : \pi_1 ( M_\tau , p_\tau (m_1)) \to \pi_1 ( N_{\tau_1} , p_{ \tau_1} (n_1))$ is an isomorphism. 
		
\noindent $(2 \Rightarrow 1)$ Suppose that there exists a homomorphism $\psi: \pi_1 ( M_\tau , p_\tau(m_1) ) \to \pi_1 ( N_{\tau_1} , p_{\tau_1}(n_1) )$ for which diagram~(\ref{diag_equiv}) is commutative. Since $p_{\tau_1} : N \to N_{\tau_1}$ is a covering map, the triviality of $\pi_i ( N , n_1)$ implies that of  $\pi_i ( N_{\tau_1} , p_{\tau_1}(n_1) ) $ for all $i \geq 2$. It follows from Theorem~\ref{set_homotopy} that there exists a map $\overline{f}: ( M_\tau , p_\tau (m_1) ) \to ( N_{\tau_1} , p_{\tau_1} (n_1) )$ such that $\overline{f}_\# = \psi$, and so by~(\ref{diag_equiv}), there exists a map $f: (M , m_1 ) \to ( N , n_1)$ that is a lift of the map $\overline{f} \circ p_\tau$ for the covering $p_{\tau_1}$, so that we have the commutative diagram~(\ref{diag_equiv_aux}). Using the short exact sequences induced by $\tau$ and $\tau_1$ in the sense of~(\ref{seq_tau}), one sees that $\alpha_\# = f_\#$ and so by Theorem~\ref{set_homotopy}, $f$ is a representative map of $\alpha$. We claim that $f$ is $(\tau, \tau_1)$-equivariant. To do so, note that for all $x \in M$, we have:
$$(p_{\tau_1} \circ f \circ \tau )( x )
= ( \overline{f} \circ p_\tau \circ \tau)( x )
= ( \overline{f} \circ p_\tau)( x )
= ( p_{\tau_1} \circ f ) ( x ).$$
Hence either $f( \tau ( x )) = f( x )$, or  $f( \tau ( x )) = \tau_1 (f ( x ))$. Let $\xi: [0,1] \to M$ be an arc from $m_1$ to $\tau(m_1)$. Then the loop $\gamma = p_\tau \circ \xi$ satisfies $\theta_\tau ( [ \gamma ]) = \overline{1}$. By~(\ref{diag_equiv}) we have:
$$\theta_{\tau_1} ( \left[ \overline{f} \circ \gamma \right] ) = \theta_{\tau_1} \circ \overline{f}_\# ( [ \gamma ] )= 
\theta_\tau ( \left[ \gamma \right] ) = \overline{1},$$
and since $ f \circ \xi$ is a lift of the loop $\overline{f} \circ \gamma$ by the covering $p_{\tau_1}$, it is an arc that is not a loop. Therefore $f( m_1) = (f \circ \xi) (0) \neq (f \circ \xi) ( 1 ) = f ( \tau ( m_1))$, and so $f ( \tau ( m_1)) = \tau_1 ( f ( m_1))$. Using standard covering space arguments and the hypothesis that $N$ is connected, it follows in a straightforward manner that the  equality  $f ( \tau (x)) = \tau_1 ( f (x))$ holds for all  $x\in M$. This proves the claim.  Finally, if $M$ and $N$ are compact surfaces without boundary and $\psi$ is an isomorphism, it follows from the classification theorem for surfaces and \cite[Theorem~5.6.2]{Zies} that $\psi$ is induced by a homeomorphism. So without loss of generality, we may take $\overline{f}$ to be a homeomorphism, and thus $f$ is also a homeomorphism.
\end{proof}
		
Let $N$ be a compact surface without boundary with base point $n_1 \in N$. Recall that $F_2 ( N ) = \{ (y_1 , y_2 ) \in N \times N \ | \ y_1 \neq y_2 \}$ is the $2\textsuperscript{nd}$  configuration space of $N$, and if $\tau_1 : F_2 ( N ) \to F_2 ( N )$ is the involution defined by $\tau_1 ( x_1 , x_2 ) = ( x_2 , x_1)$, then $D_2 ( N )$ is the associated orbit space. Let $n_2 \in N - \{ n_1 \}$, and let $ n = (n_1 , n_2 ) \in F_2 ( N)$. By~\cite[Corollary~2.2]{FadellNeu}, if $N$ is different from $\mathbb{S}^2$ and $\mathbb{RP}^2$, $F_2 ( N )$ and $D_2 ( N )$ are  manifolds for which $\pi_i ( F_2 ( N ) , n)$ and $\pi_i ( D_2 (N) , p_{\tau_1}(n) )$ are trivial for all $i \geq 2$. The groups $P_2 ( N ) = \pi_1 ( F_2 ( N) , n)$ and $B_2 ( N ) = \pi_1 ( D_2 ( N) , p_{\tau_1}(n_1) )$ are the \emph{pure} and \emph{full $2$-string braid groups} of $N$ respectively, related by the  following short exact sequence:
\begin{equation}\label{seq_braid}\xymatrix{
1 \ar[r] & P_2 ( N ) \ar[r]^-{\iota} & B_2 ( N ) \ar[r]^-{\pi} & \ztwo \ar[r] & 1,}
\end{equation}
where $\iota = ( p_{\tau_1} )_\#$ and $\pi = \theta_{\tau_1}$.

The following result generalises~\cite[Proposition 13]{GonGua} and gives an algebraic criterion in terms of braid groups to decide whether a pointed homotopy class has the Borsuk-Ulam property.

\begin{proposition}\label{borsuk_braid}
Let $(M , m_1)$ be a pointed manifold, let $\tau: M \to M$ be a free involution, and let $(N, n_1)$ be a compact surface without boundary different from $\mathbb{S}^2$ and $\mathbb{RP}^2$. For a pointed homotopy class $\alpha \in [M,m_1;N,n_1]$, the following conditions are equivalent:
\begin{enumerate}
\item $\alpha$ does not have the Borsuk-Ulam property with respect to $\tau$.
		
\item there exist homomorphisms $\varphi: \pi_1 ( M , m_1) \to P_2 (N)$ and $\psi: \pi_1 (M_\tau, p_\tau (m_1)) \to B_2 ( N ) $ for which the following diagram is commutative:

\begin{equation}\label{diag_borsuk_braid}\begin{gathered}\xymatrix{%
\pi_1 (M,m_1) \ar@{.>}[rr]^{\varphi} \ar[d]_{(p_\tau)_\#} \ar@/^0.9cm/[rrrr]^{ \alpha_\# } 
&& P_2(N) \ar[d]^-{\iota} \ar[rr]^{(p_1)_\#} && \pi_1(N,n_1) \\
\pi_1(M_\tau,p_\tau(m_1)) \ar@{.>}[rr]^{\psi} \ar[rd]_{\theta_\tau} && B_2(N) \ar[ld]^{\pi} && \\
& \ztwo ,& & &}\end{gathered}
\end{equation}
where $p_1 : F_2 ( N ) \to N$ is projection onto the first coordinate.
\end{enumerate}		
\end{proposition}

\begin{proof}
$(1 \Rightarrow 2)$ If $\alpha$ does not have the Borsuk-Ulam property, there exists a map $f: ( M , m_1) \to (N, n_1)$ such that $\alpha = [f]$ and $f( \tau (x)) \neq f(x)$ for all $x \in M$. Hence the map $F: M \to F_2 ( N )$ given by $F( x ) = ( f( x) , f( \tau(x)))$ is well defined, $(\tau, \tau_1)$-equivariant and satisfies $p_1 \circ F = f$.  Let $\varphi: \pi_1 ( M , m_1) \to P_2 ( N )$ be the homomorphism  induced  by $F$. So, we have $(p_1)_\# \circ \varphi = (p_1)_\# \circ F_\# = ( p_1 \circ F )_\# = f_\# = \alpha_\#$. By Theorem~\ref{set_homotopy} and Lemma~\ref{equiv_maps}, there exists a homomorphism $\psi: \pi_1 ( M_\tau , p_{\tau} ( m_1 )) \to B_2 ( N )$ such that $\psi \circ (p_{\tau})_\# = \iota \circ \varphi$ and $\pi \circ \psi = \theta_\tau$. This completes the first part of the proof. 

\noindent $(2 \Rightarrow 1)$ Suppose that diagram (\ref{diag_borsuk_braid}) is commutative.   Recall that $F_2 ( N)$ is a manifold which is a $K ( \pi , 1)$. So, Theorem~\ref{set_homotopy} and Lemma~\ref{equiv_maps} imply that the homomorphism $\varphi$ is induced by a $(\tau, \tau_1)$-equivariant map $F: M \to F_2 ( N)$. Let $f, g: M \to N$ be maps such that $F ( x ) = ( f( x) , g(x))$ for all $x \in M$. Since $F$ is equivariant, we have $f( \tau (x)) = g(x) \neq f(x)$ for all $x \in M$. Again, by~(\ref{diag_borsuk_braid}) and Theorem~\ref{set_homotopy} we have $\alpha = [f]$, and thus $\alpha$ does not have the Borsuk-Ulam property.
\end{proof}

The following result shows that to solve the Borsuk-Ulam problem for free homotopy classes, it suffices to solve it for pointed homotopy classes.

\begin{proposition}\label{borsuk_pointed_free}
Suppose that the hypotheses of Proposition \ref{borsuk_braid} hold. 
\begin{enumerate}
\item Let $\alpha, \alpha' \in [ M , m_1 ; N , n_1]$ and suppose that  the homomorphisms $ \alpha_\#  ,  \alpha^\prime_\#  : \pi_1 ( M , m_1 ) \to \pi_1 ( N , n_1 ) $ are equivalent.   Then $\alpha$ does not have the Borsuk-Ulam property with respect to $\tau$ if and only if $\alpha'$ does not have the Borsuk-Ulam property with respect to $\tau$.
	
\item Let $\beta\in [M , N]$. Then there exists   $\alpha  \in [ M , m_1 ; N , n_1]$  such that  $\beta =  \alpha_\free$. Further, $\beta$ does not have the Borsuk-Ulam property with respect to $\tau$ if and only if $\alpha$ does not have the Borsuk-Ulam property with respect to $\tau$.
\end{enumerate}
\end{proposition}
		
\begin{proof}
To prove part~1, first note that by the symmetry of the statement with respect to $\alpha$ and $\alpha'$, it suffices to prove one of the implications of the conclusion. So suppose that $\alpha$ does not have the Borsuk-Ulam property with respect to $\tau$. By Proposition~\ref{borsuk_braid}, there exist homomorphisms $\varphi: \pi_1 ( M , m_1) \to P_2 (N)$ and $\psi: \pi_1 (M_\tau, p_\tau (m_1)) \to B_2 ( N ) $ such that diagram (\ref{diag_borsuk_braid}) is commutative. Since the homomorphism $(p_1)_\#: P_2 ( N ) \to \pi_1 (N , n_1)$ is surjective by~\cite[Theorem 1.4]{Bir}, and the homomorphisms $ \alpha_\# $ and $ \alpha^\prime_\# $ are equivalent  via  an element of $\pi_1 ( N ,n_1)$, $\gamma$ say, there exists $b \in P_2 ( N)$ such that $ (p_1)_\# ( b ) = \gamma$. If $\varphi^\prime : \pi_1 ( M , m_1 ) \to P_2 ( N )$ (resp.\ $\psi^\prime: \pi_1 (M_\tau, p_\tau (m_1)) \to B_2 ( N ) $) is the homomorphism given by $ \varphi'(v)=b\varphi(v)b^{-1}$ for all $v \in  \pi_1 ( M , m_1 )$ (resp.\ $\psi^\prime(w)=\iota(b)\psi(w)\iota (b)^{-1}$ for all $w\in\pi_1 (M_\tau, p_\tau (m_1))$)  where $\iota$ is as in~(\ref{seq_braid}), then diagram~(\ref{diag_borsuk_braid}) remains commutative if we replace $\alpha_\#$, $\varphi$ and $\psi$ by $\alpha_\#^{\prime}$, $\varphi^{\prime}$ and $\psi^\prime$. It follows from Proposition~\ref{borsuk_braid} that $\alpha^\prime$ does not have the Borsuk-Ulam property with respect to $\tau$.

To prove part~2, if   $\beta\in [M , N]$,  by  Theorem~\ref{set_homotopy} there exists $\alpha  \in [ M , m_1 ; N , n_1]$  such that  $\beta =  \alpha_\free$.  Suppose that $\beta$ does not have the Borsuk-Ulam property with respect to $\tau$. So there exists a map $f: M \to N$ such that $\beta = [f]$ and $f( \tau ( x) ) \neq f(x)$ for all $x \in M$. By~\cite[Lemma 6.4]{Vick}, there exists a homeomorphism $H: N \to N$  that is homotopic to the identity such that $H ( f ( m_1) ) = n_1$. Thus the pointed homotopy class $\alpha^\prime = [H \circ f] \in [M , m_1 ; N , n_1]$ does not have the Borsuk-Ulam property with respect to $\tau$. By equations~(\ref{defGammaMN}) and~(\ref{defLambdaMN}) and Theorem~\ref{set_homotopy} it follows that: 
\begin{align*}
\Upsilon_{M,N} ( \alpha_\# )
& = ( \Upsilon_{M,N} \circ \Gamma_{M,N} ) ( \alpha )
= ( \Delta_{M,N} \circ \Lambda_{M,N} )( \alpha )
= \Delta_{M,N}( \alpha_\free ) \\
& = \Delta_{M,N} ( \beta ) 
= \Delta_{M,N} ( [ f ] )
= \Delta_{M,N}( [ H \circ f ] ) 
= \Delta_{M,N} ( ( \alpha^\prime )_\free ) \\
& = ( \Delta_{M,N} \circ \Lambda_{M,N} )( \alpha^\prime ) 
= ( \Upsilon_{M,N} \circ \Gamma_{M,N} )( \alpha^\prime )
= \Upsilon_{M,N} ( \alpha^\prime_\# ).
		\end{align*}
Hence the homomorphisms $\alpha_\#$ and $\alpha^\prime_\#$ are equivalent  by~(\ref{defUpsilonMN}). We conclude from part~1 that $\alpha$ does not have the Borsuk-Ulam property with respect to $\tau$. The converse is obvious.
\end{proof}
		
\begin{remark}
Taken together, Theorem~\ref{set_homotopy} and Propositions~\ref{borsuk_braid} and~\ref{borsuk_pointed_free} provide an algebraic criterion to solve the Borsuk-Ulam problem for free  homotopy classes.
\end{remark}

To end this section, we discuss briefly which involutions of compact surfaces we should consider in order to solve the Borsuk-Ulam problem for homotopy classes. As in~\cite[Corollary 2.3]{GonHaZe} if $\tau_1, \tau_2 : M \to M$ are free involutions, we say that $\tau_1$ and $\tau_2$ are \emph{equivalent}, written $\tau_1\sim\tau_2$, if there exists a $(\tau_1, \tau_2)$-equivariant homeomorphism $H: M \to M$. 

\begin{proposition}\label{reduce_involution}
Let $M$ and $N$ be topological spaces, and let $\tau_1, \tau_2 : M \to M$ be equivalent free involutions. Let $H : M \to M$ be a $(\tau_1 , \tau_2)$-equivariant homeomorphism.  Then the map $\mathcal{H} : [M , N] \to [M , N]$ defined by $\mathcal{H}( [f] ) = [f \circ H^{-1}]$ is a bijection. Further, a homotopy class $\beta\in [M,N]$ does not have the Borsuk-Ulam property with respect to $\tau_1$ if and only if $\mathcal{H} ( \beta )$ does not have the Borsuk-Ulam property with respect to $\tau_2$.
\end{proposition}

\begin{proof} 
Let $\tau_1 , \tau_2 : M \to M$ be free involutions, and suppose that $H: M \to M$ is a $(\tau_1, \tau_2)$-equivariant homeomorphism and that $\beta\in [M , N]$. If $\beta$ does not have the Borsuk-Ulam property with respect to $\tau_1$, there exists a map $f_1: M \to N$ such that $[f_1]=\beta$ and $f_1( \tau_1 ( x )) \neq f_1(x)$ for all $x \in M$. So the map $f_2 = f_1 \circ H^{-1}: M \to N$ satisfies $[f_2]=\mathcal{H} ( \beta )$ and $f_2 ( \tau_2 (x)) \neq f_2( x)$ for all $x \in M$, in other words $\mathcal{H} ( \beta )$ does not have the Borsuk-Ulam property with respect to $\tau_2$. The converse follows in a similar manner.
\end{proof}
		
Now let:
\begin{equation*}
I=\{\tau: M \to M \, |\,  \text{$\tau$ is a free involution, $M$ is a compact surface without boundary, $M\neq \mathbb{S}^2,\mathbb{RP}^2$}\},
\end{equation*}
and let $\mathcal{I} = I/\!\!\sim$ be the corresponding quotient set. We give an algebraic interpretation of this definition. Let $E$ denote the set of all short exact sequences of the form $1 \to A \stackrel{i}{\to} B \stackrel{j}{\to} \ztwo \to 1$, where $A$ and $B$ are fundamental groups of compact surfaces without boundary different from $\mathbb{S}^2$ and $\mathbb{RP}^2$. If $k \in \{ 1 , 2 \}$, and $S_k$ is an element of $E$ of the form $1 \to A_k \stackrel{i_k}{\to} B_k \stackrel{j_k}{\to} \ztwo \to 1$, we say that $S_{1}$ and $S_{2}$ are \emph{equivalent}, written $S_{1} \approx S_{2}$, if there exist isomorphisms $\varphi : A_1 \to A_2$ and $\psi : B_1 \to B_2$ such that the following diagram is commutative:
$$\xymatrix{%
1 \ar[r] & A_1 \ar[r]^-{i_1} \ar[d]^-{\varphi} & B_1 \ar[r]^{j_1} \ar[d]^-{\psi} & \ztwo \ar[r] \ar[d]^-{\text{Id}} &  1 \\
1 \ar[r] & A_2 \ar[r]^-{i_2} & B_2 \ar[r]^{j_2} & \ztwo \ar[r] &  1,}$$
where $\text{Id}$ denotes the identity. It is easy to see that $\approx$ is a equivalence relation on $E$. Let $\mathcal{E} = E/\!\!\approx$ be the corresponding quotient set. Then we have the following proposition. 

\begin{proposition}\label{invo_seq_exact}
The map  $I \longrightarrow E$ given by~(\ref{seq_tau}) that to a free involution $\tau$ associates the short exact sequence $S_{\tau}$ gives rise to a bijection $\Sigma: \mathcal{I} \to \mathcal{E}$.		
\end{proposition}

\begin{proof}
By Lemma~\ref{equiv_maps}, $\Sigma$ is well defined and injective. It remains to show that $\Sigma$ is surjective. Let $S\in E$ be a short exact sequence of the form $1 \to A \stackrel{i}{\to} B \stackrel{j}{\to} \ztwo \to 1 $, let $M$ and $N$ be compact surfaces without boundary such that $A = \pi_1 ( M )$ and $B = \pi_1 ( N)$, and let $\hat N$ be the double covering of $N$ determined by the subgroup $A\subset B$. Since $M$ and $\hat N$ have isomorphic fundamental groups, there exists a homeomorphism $h:M \to \hat N$ that realises a given isomorphism between the fundamental groups of $M$ and $\hat N$. Then the composition of $h $ with the double covering $\hat N \to N$ is a double covering of $N$ that determines a free involution $\tau$ on $M$, and $\Sigma [  \tau  ] = [ S ]$ as required.
\end{proof}

\section{The braid groups of the torus}\label{braid_torus}

The aim of this section is to prove Theorem~\ref{remark_p2_torus} that provides  an explicit isomorphism between $P_2(\torus)$ and $F(x,y)\oplus\zsz$ and describes the action of $\sigma$ on $P_2(\torus)$. 
		
If $G$ is a group and $a,b \in G$, let $[ a , b ] = a b a^{-1} b^{-1}$ denote their commutator. The following presentation of $P_2( \torus )$ may be found in~\cite[Section 4]{FadHus}.

\begin{theorem}\label{presentation_p2_t2}
The following is a presentation of $P_2( \torus )$:

\noindent generators: $\rho_{1,1}, \rho_{1,2}, \rho_{2,1}, \rho_{2,2}, B$.

\noindent relations:

\begin{enumerate}[(i)]
\item\label{it:presentation_p2_t2a} $[ \rho_{1,1} , \rho_{1,2}^{-1} ] = [ \rho_{2,1}, \rho_{2,2}^{-1} ] = B$.
		
\item\label{it:presentation_p2_t2b} $\rho_{2,k} \rho_{1,k} \rho_{2,k}^{-1} = B \rho_{1,k} B^{-1}$  and 
$\rho_{2,k}^{-1} \rho_{1,k} \rho_{2,k} = \rho_{1,k} [B^{-1}, \rho_{1,k}]$ for all $k \in \{ 1,2 \}$.

\item\label{it:presentation_p2_t2c} $\rho_{2,1} \rho_{1,2} \rho_{2,1}^{-1} = B \rho_{1,2} [ \rho_{1,1}^{-1} , B ]$ and
$\rho_{2,1}^{-1} \rho_{1,2} \rho_{2,1} = B^{-1} [ B, \rho_{1,1} ] \rho_{1,2} [ B^{-1} , \rho_{1,1} ]$.

\item\label{it:presentation_p2_t2d} $\rho_{2,2} \rho_{1,1} \rho_{2,2}^{-1} = \rho_{1,1} B^{-1}$ and 
$\rho_{2,2}^{-1} \rho_{1,1} \rho_{2,2} = \rho_{1,1} B [ B^{-1} , \rho_{1,2} ]$.
		\end{enumerate}
		\end{theorem}

Note that $B$ may be removed from the list of generators using relation~(\ref{it:presentation_p2_t2a}). In order to obtain a presentation of $P_2 ( \torus )$ that is more suitable for our purposes, we shall use the following lemma. 

\begin{lemma}\label{center_p2_torus}
With respect to the presentation of $P_2( \torus )$ given in Theorem~\ref{presentation_p2_t2}, the following relations hold:
\begin{enumerate}[(i)]\setcounter{enumi}{4}
\item\label{it:presentation_p2_t2e} $\rho_{2,k} B \rho_{2,k}^{-1} = B \rho_{1,k}^{-1} B \rho_{1,k} B^{-1}$ for all $k \in \{ 1,2 \}$.
	
\item\label{it:presentation_p2_t2f} $\rho_{1,k} B^{-1} \rho_{2,k} \rho_{i,j} \rho_{2,k}^{-1} B \rho_{1,k}^{-1} = \rho_{i,j}$ for all $i,j,k \in \{ 1 , 2 \}$.
\end{enumerate}
\end{lemma}

Observe that relation~(\ref{it:presentation_p2_t2f}) implies that for $k\in \{1,2\}$, $\rho_{1,k} B^{-1} \rho_{2,k}$ belongs to the centre of $P_2( \torus )$. 

\begin{proof}[Proof of Lemma~\ref{center_p2_torus}]
We have:
\begin{align*}
\rho_{2,1} B \rho_{2,1}^{-1} &\stackrel{\text{(\ref{it:presentation_p2_t2a})}}{=}  [ \rho_{2,1} \rho_{1,1} \rho_{2,1}^{-1} , \rho_{2,1} \rho_{1,2}^{-1} \rho_{2,1}^{-1} ]\stackrel{\text{(\ref{it:presentation_p2_t2b}),(\ref{it:presentation_p2_t2c})}}{=} \rho_{1,1} \rho_{1,2}^{-1} \rho_{1,1}^{-1} \rho_{1,2} \rho_{1,1}^{-1} B \rho_{1,1} B^{-1} \stackrel {\text{(\ref{it:presentation_p2_t2a})}}{=}  B \rho_{1,1}^{-1} B \rho_{1,1} B^{-1},\; \text{and}\\
\rho_{2,2} B \rho_{2,2}^{-1} &\stackrel{\text{(\ref{it:presentation_p2_t2a})}}{=}  [ \rho_{2,2} \rho_{1,1} \rho_{2,2}^{-1} , \rho_{2,2} \rho_{1,2}^{-1} \rho_{2,2}^{-1} ] \stackrel{\text{(\ref{it:presentation_p2_t2b}),(\ref{it:presentation_p2_t2d})}}{=}  \rho_{1,1} \rho_{1,2}^{-1} \rho_{1,1}^{-1} B \rho_{1,2} B^{-1} 
\stackrel{\text{(\ref{it:presentation_p2_t2a})}}{=}   B \rho_{1,2}^{-1} B \rho_{1,2} B^{-1},
\end{align*}
which proves~(\ref{it:presentation_p2_t2e}). To prove~(\ref{it:presentation_p2_t2f}), first let $i = 1$. If $j = k$, the relation follows directly from~(\ref{it:presentation_p2_t2b}). If $j\neq k$, we have:
\begin{align*}
\rho_{1,2} B^{-1} \rho_{2,2} \rho_{1,1} \rho_{2,2}^{-1} B \rho_{1,2}^{-1} &\stackrel{\text{(\ref{it:presentation_p2_t2d})}}{=} 
 \rho_{1,2} B^{-1} \rho_{1,1} \rho_{1,2}^{-1}  \stackrel{\text{(\ref{it:presentation_p2_t2a})}}{=} \rho_{1,1} \;\text{and}\\
 \rho_{1,1} B^{-1} \rho_{2,1} \rho_{1,2} \rho_{2,1}^{-1} B \rho_{1,1}^{-1} &\stackrel{\text{(\ref{it:presentation_p2_t2c})}}{=} 
 \rho_{1,1} \rho_{1,2} \rho_{1,1}^{-1} B \stackrel{\text{(\ref{it:presentation_p2_t2a})}}{=} \rho_{1,2}.
\end{align*}
Suppose that $i = 2$. If $j=k$ then:
\begin{equation*}
\rho_{1,k} B^{-1} \rho_{2,k} \rho_{2,k} \rho_{2,k}^{-1} B \rho_{1,k}^{-1} =\rho_{1,k} B^{-1} \rho_{2,k} B \rho_{1,k}^{-1} \rho_{2,k}^{-1} \rho_{2,k}   \stackrel{\text{(\ref{it:presentation_p2_t2b}),(\ref{it:presentation_p2_t2e})}}{=} \rho_{2,k}.
\end{equation*}
Now suppose that $j = 1$ and $k = 2$. By~(\ref{it:presentation_p2_t2c}) and~(\ref{it:presentation_p2_t2e}), we have $\rho_{2,1}B\rho_{1,2}^{-1}\rho_{2,1}^{-1}=B\rho_{1,2}^{-1}B^{-1}$, and thus
		$$
\rho_{1,2} B^{-1} \rho_{2,2} \rho_{2,1} \rho_{2,2}^{-1} B \rho_{1,2}^{-1} =
\rho_{1,2} B^{-1} \rho_{2,2} \rho_{2,1} \rho_{2,2}^{-1} \rho_{2,1}^{-1} B \rho_{1,2}^{-1} B^{-1} \rho_{2,1} 
\stackrel{\text{(\ref{it:presentation_p2_t2a})}}{=}  \rho_{1,2} B^{-1} \rho_{2,2} B \rho_{2,2}^{-1} B \rho_{1,2}^{-1} B^{-1} \rho_{2,1}
\stackrel{(v)}{=} \rho_{2,1}.$$
Finally, for $j = 2$ and $k = 1$, we have:
\begin{align*}
\rho_{1,1} B^{-1}	\rho_{2,1} \rho_{2,2} \rho_{2,1}^{-1} B \rho_{1,1}^{-1} \stackrel{\text{(\ref{it:presentation_p2_t2a})}}{=} &
\rho_{1,1} B^{-1} \rho_{2,2} \rho_{1,1}^{-1} \stackrel{\text{(\ref{it:presentation_p2_t2d})}}{=} \rho_{2,2}.	\qedhere
\end{align*}
\end{proof}

By~\cite[Theorem 1]{FadellNeu}, the projection $p_1 : F_2 ( \torus ) \to \torus$ onto the first coordinate  is a locally-trivial fibre space whose fibre may be identified with $\torus - \{ * \}$. This gives rise to the following short exact sequence:
\begin{equation}\label{seq_fad_neu}\xymatrix{%
1 \ar[r] & \pi_1 ( \torus - \{ * \} ) \ar[r] & P_2 ( \torus ) \ar[r]^-{ (p_1)_\# } & \pi_1 ( \torus ) \ar[r] & 1.
}\end{equation}

\begin{remark}\label{identification_pi1_torus}
By~\cite[Figure 4.1]{FadHus}, $\pi_1 ( \torus - \{ * \} )$ is a free group generated by $\rho_{2,1}$ and $\rho_{2,2}$, which we denote by $F( \rho_{2,1} , \rho_{2,2})$. The elements $(p_1)_\# ( \rho_{1,1} )$ and $(p_1)_\# ( \rho_{1,2} )$ generate the free Abelian group $\pi_1 ( \torus )$ that we shall identify
with $\zsz$ under the  correspondence of $(p_1)_\# ( \rho_{1,1} )$ (resp.\ $(p_1)_\# ( \rho_{1,2} )$) with $(1,0)$ (resp.\ $(0,1)$).
\end{remark}
		
By Lemma~\ref{center_p2_torus}(\ref{it:presentation_p2_t2f}), the map $\varphi : \pi_1 ( \torus ) \to P_2 ( \torus )$ defined on the generators of $\pi_1 ( \torus )$ by $\varphi ( 1,0 ) = \rho_{1,1} B^{-1} \rho_{2,1}$ and $\varphi ( 0,1 ) = \rho_{1,2} B^{-1} \rho_{2,2} $ is a homomorphism that may be seen to be a section for $(p_1)_\#$ using Remark~\ref{identification_pi1_torus}. By the short exact sequence~(\ref{seq_fad_neu}), we conclude that:
\begin{equation}\label{decomposition_p2_torus}
P_2 ( \torus ) \cong F ( \rho_{2,1} , \rho_{2,2}) \oplus \z [ \rho_{1,1} B^{-1} \rho_{2,1} ] \oplus \z [ \rho_{1,2} B^{-1} \rho_{2,2} ]. 	
\end{equation}
In particular, the centre of $P_{2}(\torus)$ is generated by $\rho_{1,1} B^{-1} \rho_{2,1}$ and $\rho_{1,2} B^{-1} \rho_{2,2}$. Let $\sigma$ be the standard generator of $B_2 ( \torus )$ that swaps the two base points, so $B = \sigma^2$. By~\cite[Section 5]{GonGua}, the automorphism $l_\sigma: P_2 ( \torus ) \to P_2 ( \torus )$ given by conjugation by $\sigma$ satisfies $l_\sigma ( \rho_{1,k} ) = \rho_{2,k}$ and $l_\sigma ( \rho_{2,k} ) = B \rho_{1,k} B^{-1}$ for all $k \in \{ 1 , 2 \}$. Using the decomposition given in~(\ref{decomposition_p2_torus}), for all $k  \in \{ 1 , 2 \}$, we have 
$l_\sigma ( \rho_{2,k} ) = B \rho_{2,k}^{-1} ( \rho_{1,k} B^{-1} \rho_{2,k})$ and $l_\sigma ( \rho_{1,k} B^{-1} \rho_{2,k}) = \rho_{1,k} B^{-1} \rho_{2,k}$. By~(\ref{decomposition_p2_torus}) and the fact that $B_{2}(\torus)$ is generated by $\{\sigma, \rho_{1,1}, \rho_{1,2}, \rho_{2,1}, \rho_{2,2}\}$, this implies that the centre of $B_{2}(\torus)$ is generated by $\rho_{1,1} B^{-1} \rho_{2,1}$ and $\rho_{1,2} B^{-1} \rho_{2,2}$. Writing $x = \rho_{2,1}$ and $y = \rho_{2,2}$ and identifying $\z [ \rho_{1,1} B^{-1} \rho_{2,1} ] \oplus \z [ \rho_{1,2} B^{-1} \rho_{2,2}  ]$ with $\zsz$, we may summarise the main algebraic properties of $P_2 ( \torus )$ as follows.
    
\begin{theorem}\label{remark_p2_torus}
Up to isomorphism, $P_2 ( \torus )$ may be written in the form $F( x , y ) \oplus \zsz$, and relative to this identification:
\begin{enumerate}[(a)]
\item the homomorphism $(p_1)_\# : P_2 ( \torus ) \to \pi_1 ( \torus )$ is projection onto the group $\zsz$, in other words, $(p_1)_\# ( w , m , n) = ( m , n)$ for all $x\in F( x , y )$ and $m,n\in \mathbb{Z}$.
		
\item the element $\sigma$ belongs to $B_2 ( \torus ) \backslash P_2 ( \torus )$ and satisfies $\sigma^2 = (B , 0 , 0)$, where $B=[x,y^{-1}]$.
	
\item the images of the generators of $P_2 ( \torus )$ under the homomorphism $l_\sigma: P_2 ( \torus ) \to P_2 ( \torus )$ given by conjugating by $\sigma$ are as follows: 
\begin{equation*}
\text{$l_\sigma( x , 0 , 0) = (B x^{-1}, 1 , 0)$, $l_\sigma( y , 0 , 0) = (B y^{-1}, 0 , 1)$,   $l_\sigma( \id , 1 , 0 ) = ( \id , 1 , 0 )$ and  $l_\sigma( \id , 0 , 1 ) = ( \id , 0 , 1 )$,}
\end{equation*}
where $\id$ denotes the trivial element of $F(x,y)$.
\end{enumerate}
\end{theorem}

The following equation will be used for various computations and arguments in Section~\ref{borsuk_maps_torus}. Let $| \cdot |_x , | \cdot |_y: F( x , y) \to \z$ be the exponent sum homomorphisms that are defined on the generators of $F( x , y)$ by $| x |_x = | y |_y = 1$ and $| x |_y = | y |_x = 0$.  By Theorem~\ref{remark_p2_torus}, we have $l_\sigma ( x , 0 , 0 ) = ( ( x^{-1} y ) x^{-1} (x^{-1} y)^{-1} , 1 , 0)$ and $l_\sigma(y , 0 , 0 ) = ( ( x^{-1} y ) y^{-1} (x^{-1} y)^{-1} , 0 , 1)$. From these, if $w = w(x,y) \in F(x,y)$, we obtain the following equality: 
\begin{equation}\label{formula_abelian}
l_\sigma ( w(x,y) , 0 , 0 ) = \left( xy^{-1} w( x^{-1} , y^{-1} ) y x^{-1} , |w|_x , |w|_y \right).		
\end{equation}

\section{The braid groups of the Klein bottle}\label{braid_klein}

In this section, we give a presentation of $P_{2}(\klein)$ in Theorem~\ref{presentation_p2_klein}, and we give an explicit description of the semi-direct product structure of this group in Theorem~\ref{braid_group_klein}. We then determine the action by conjugation of $\sigma$ on the elements of $P_{2}(\klein)$ in Theorem~\ref{action_sigma_k2} that will be used in later sections.

If $G$ is a group and $a,b \in G$, we define their \emph{anti-commutator} by $[ a , b ]^\prime = a b a b^{-1}$. By \cite[Theorem A.3]{Bel}, $B_2 ( \klein)$ has the following presentation.
		
\begin{theorem}\label{presentation_belin}
The following is a presentation of $B_2( \klein)$:

\noindent generators:  $a_1$, $a_2$, $\sigma$.

\noindent relations:

\begin{enumerate}
\item[(R2)] $\sigma^{-1} a_r \sigma^{-1} a_r = a_r \sigma^{-1} a_r \sigma$, for all $r \in \{ 1,2 \}$.
\item[(R3)] $\sigma^{-1} a_1 \sigma a_2 = a_2 \sigma^{-1} a_1 \sigma$.
\item[(TR)] $a_1^2 a_2^2 = \sigma^2$.
\end{enumerate}
\end{theorem}

In order to exhibit a presentation of $P_2 ( \klein )$, we apply the Reidemeister-Schreier rewriting process that is described in detail in \cite[Chapter 2, Theorem 2.8]{Magnus} and briefly in~\cite[Appendix I, Theorem 6.3]{Mura}. We use the notation of~\cite{Mura}.
	
\begin{theorem}\label{presentation_p2_klein}
The following is a presentation of $P_2 ( \klein )$:

\noindent generators:  $a_1$, $a_2$, $b_1$, $b_2$, $B$.

\noindent relations:

\begin{enumerate}[(i)]
\item  $b_r a_r = B a_r B^{-1} b_r B$, for all $r \in \{1 , 2 \}$.
\item $[a_r, b_r] = a_r B a_r^{-1}$, for all $r \in \{1 , 2 \}$.
\item $b_1 B a_2 B^{-1} = B a_2 B^{-1} b_1$.
\item $[a_1 , b_2] = 1$.
\item $a_1^2 a_2^2 = b_1^2 b_2^2 = B$.
\end{enumerate}
The inclusion $\iota: P_2 ( \klein) \to B_2 ( \klein)$ is defined on the generators of $P_2 ( \klein)$ by $\iota ( a_r ) = a_r$, $\iota ( b_r ) = \sigma a_r \sigma^{-1}$, $r \in \{ 1 , 2 \}$, and $\iota ( B ) = \sigma^2$.
\end{theorem}
		
\begin{proof}
First note that $\{ \id , \sigma \}$ is a Schreier system of $P_2 ( \klein )$. Let us compute the generators of $P_2 ( \klein )$. For $r \in \{ 1 , 2 \}$, we have:
\begin{multicols}{2}
\noindent $\varrho( \id , a_r) = \id a_r \overline{\id a_r}^{-1} = a_r$.

\noindent $\varrho( \id , \sigma ) = \id \sigma \overline{ \id \sigma}^{-1} = \id$.

\noindent $\varrho(\sigma, a_r) = \sigma a_r \overline{\sigma a_r}^{-1} = \sigma a_r \sigma^{-1} = b_r$.

\noindent $\varrho(\sigma, \sigma) = \sigma \sigma \overline{\sigma \sigma}^{-1} = \sigma^2 = B$.
		\end{multicols}
\noindent To determine the relations, we apply the Reidemeister-Schreier rewriting process. We leave the straightforward calculations to the reader.
\end{proof}

Recall that the projection $p_1 : F_2 ( \klein ) \to \klein$ onto the first coordinate is a locally-trivial fibre space whose fibre may be identified with $\klein - \{ * \}$~\cite[Theorem~1]{FadellNeu}, from which we obtain the following short exact sequence:
\begin{equation}\label{seq_fad_klein}\xymatrix{
1 \ar[r] & \pi_1 ( \klein - \{ * \} ) \ar[r] & P_2 ( \klein ) \ar[r]^-{ (p_1)_\# } & \pi_1 ( \klein ) \ar[r] & 1.
}\end{equation}
Using~\cite[Figure 10]{Bel} and the presentation of $P_2 ( \klein )$ given by Theorem~\ref{presentation_p2_klein}, we conclude that $\pi_1 ( \klein - \{ * \} )$ is the free group generated by $\{b_{1},b_{2}\}$, which we denote by $F( b_1 , b_2)$. For $k \in \{ 1 , 2 \}$, let $\widetilde{a}_k = ( p_1)_\# ( a_k )$. By Theorem~\ref{presentation_p2_klein} and the short exact sequence~(\ref{seq_fad_klein}), we have $\pi_1 ( \klein ) = \left\langle \widetilde{a}_1 , \widetilde{a}_2 \ | \ \widetilde{a}_1^2 \widetilde{a}_2^2 = \id \right\rangle$. Let $u_1 = \widetilde{a}_1 \widetilde{a}_2$ and $v_1 = \widetilde{a}_2^{-1}$. Then $\widetilde{a}_1 = u_1 v_1$ and $\widetilde{a}_2 = v_1^{-1}$. This choice of elements yields another classical presentation of the fundamental group of the Klein bottle of the form
$\pi_1 ( \klein ) 
= \left\langle u_1 , v_1 \ | \ [ u_1 , v_1 ]^\prime = \id \right\rangle
= \z [ u_1 ] \rtimes \z [ v_1 ]$.

\begin{remark}\label{identification_pi1_klein}
As a set, the group $\zsdz$ is the Cartesian product of $\z$ with itself. The group operation is given by $(r_1 , s_1)(r_2,s_2) = (r_1 +(-1)^{s_1} r_2 , s_1 + s_2)$, the identity element is $(0,0)$, the inverse of an element $(r,s)$ is $((-1)^{s+1} r, -s)$, and there is a canonical isomorphism between $\z [ u_1 ] \rtimes \z [ v_1 ] $ and $\zsdz$. In the rest of this paper, we identify $\pi_1 ( \klein ) $ with $ \zsdz$ in this way.
\end{remark}

Since $\klein$ admits a non-vanishing vector field, the projection $p_{1}$ admits a geometric section~\cite{FadellNeu}, and this implies that the short exact sequence~(\ref{seq_fad_klein}) splits. In the following proposition, we give an explicit algebraic section.

\begin{proposition}\label{section_p1}
The correspondences $\varphi( 1 , 0) = b_1 a_1 b_2 a_2$ and $\varphi ( 0 , 1 ) = a_2^{-1} b_2^{-1}$ define a homomorphism $\varphi : \pi_1 ( \klein ) \to P_2 ( \klein )$ that is a section for the homomorphism $(p_1)_\# : P_2 ( \klein ) \to \pi_1 ( \klein)$.
\end{proposition}

\begin{proof}
We have that:
\begin{align*}
[ \varphi(1,0) , \varphi(0,1) ]^\prime  
& = b_1 a_1 b_1 a_1 b_2 a_2 b_2 a_2 
= [ b_1 , a_1 ] a_1 b_1^2 a_1 b_2 a_2 [ b_2 , a_2 ] a_2 b_2 \\
& \stackrel{\text{(ii)}}{=} a_1 B^{-1} b_1^2 a_1 b_2 a_2^2 B^{-1} b_2 
\stackrel{\text{(v)}}{=} a_1 b_2^{-2} a_1 b_2 a_1^{-2} b_2
\stackrel{\text{(iv)}}{=}  \id.
\end{align*}
where we make use of the relations given in Theorem~\ref{presentation_p2_klein}. So by Remark~\ref{identification_pi1_klein}, $\varphi$ is a well-defined homomorphism. Using the relations 
$\widetilde{a}_1 = u_1 v_1$ and $\widetilde{a}_2 = v_1^{-1}$, it follows that $\varphi$ is a section for $(p_1)_\#$.
\end{proof}

\begin{remark}\label{seq_split}
Given a split short exact sequence
$\xymatrix{
1 \ar[r] & A \ar[r] & G \ar[r]^-\pi & B \ar[r] \ar@/^0.3cm/[l]^-\varphi & 1,		
}$
recall that there exists a homomorphism $\theta: B \to \operatorname{\text{Aut}}( A )$, defined by $\theta (b)(a) = \varphi(b) a \varphi(b)^{-1}$, and isomorphisms 
$\xymatrix{ A \rtimes_\theta B  \ar@<0.1cm>[r]^-\lambda 
& G \ar@<0.1cm>[l]^-\gamma}$, defined by $\lambda ( a , b ) = a \varphi ( b)$ and $\gamma ( g ) = (g \varphi( \pi( g ))^{-1}, \pi(g))$.
\end{remark}
		
Proposition~\ref{section_p1} and Remark~\ref{seq_split} imply that $P_2 ( \klein ) \cong F ( b_1 , b_2 ) \rtimes_\theta ( \zsdz )$. In the following result,  we 
describe the action by conjugation of $a_r$ on the elements $b_k$ and $B$ for all $r, k \in \{ 1 , 2 \}$.

\begin{lemma}\label{action_a1_a2}
In $P_2 ( \klein )$, the following relations hold:
\begin{enumerate}[(1)]
\begin{multicols}{2}
\item $a_1 b_1 a_1^{-1} = b_1^{-1} b_2^{-2}$.
	
\item $a_1 b_2 a_1^{-1} = b_2$.
	
\item $a_1 B a_1^{-1} = b_1^{-1} B^{-1} b_1$.
	
\item $a_2 b_1 a_2^{-1} = b_2^{-1} B^{-1} b_2^{-1} b_1 b_2 B b_2$.
	
\item $a_2 b_2 a_2^{-1} = b_2^{-1} B^{-1} b_2^2$.
	
\item $a_2 B a_2^{-1} = b_2^{-1} B^{-1} b_2$.
\end{multicols}
\end{enumerate}
\end{lemma}
		
\begin{proof}
First, we have:
$$
(a_1 b_1 a_1)^2
\stackrel{\text{(v)}}{=} a_1 B b_2^{-2} a_1^{-1}
\stackrel{\text{(iv)}}{=} a_1 B a_1^{-1} b_2^{-2}
\stackrel{\text{(ii)}}{=} (a_1 b_1 a_1^{-1}) b_1^{-1} b_2^{-2}.
$$
This equality implies relation~(1). Relation~(2) follows directly from~(iv). Further:
$$
a_1 B a_1^{-1}
\stackrel{\text{(v)}}{=} a_1 b_1^2 b_2^2 a_1^{-1}
\stackrel{\text{(iv)}}{=} (a_1 b_1 a_1^{-1})^2 b_2^2
\stackrel{\text{(1)}}{=} b_1^{-1} b_2^{-2} b_1^{-1}
\stackrel{\text{(v)}}{=} b_1^{-1} B^{-1} b_1,
$$
which proves relation~(3). It remains to prove relations~(4),~(5) and~(6). We have:
\begin{align}
b_1 B a_2 B^{-1} a_2^{-1}
&\stackrel{\text{(iii)}}{=} B a_2 B^{-1} b_1 a_2^{-1} = B (a_2 B^{-1} a_2^{-1} ) ( a_2 b_1 a_2^{-1} ), \;\text{which implies that}\notag\\
a_2 b_1 a_2^{-1} &= a_2 B a_2^{-1} B^{-1} a_2 B^{-1} a_2^{-1},\label{conj_a2_b1}
\end{align}
and
\begin{equation}\label{conj_a2_b2}
a_2 b_2 a_2^{-1} = [a_2 , b_2 ] b_2 \stackrel{\text{(ii)}}{=} a_2 B a_2^{-1} b_2.		
\end{equation}
Using~(\ref{conj_a2_b1}) and~(\ref{conj_a2_b2}), we obtain:
$$
a_2 B a_2^{-1} 
\stackrel{\text{(v)}}{=} (a_2 b_1 a_2)^2 (a_2 b_2 a_2)^2
= a_2 B a_2^{-1} B^{-1} b_1^2 B b_2 a_2 B a_2^{-1} b_2
\stackrel{\text{(v)}}{=} ( a_2 B a_2^{-1} ) b_2^{-2} B b_2 (a_2 B a_2^{-1} ) b_2,
$$
which implies relation~(6). Using relation~(6), relations~(4) and~(5) then follow from~(\ref{conj_a2_b1}) and~(\ref{conj_a2_b2}) respectively.
\end{proof}

To simplify various computations, it will be convenient to carry out the change of basis given by $u= b_1b_2$, $v=b_2^{-1}$, so $b_1=u v$ and $b_2 = v^{-1}$. Note that in the new basis, $B = [ u , v ]^\prime$. In order to determine the automorphism $\theta ( m , n ) : F ( u , v ) \to F ( u , v)$ for each $(m,n) \in \zsdz$, we require the following two lemmas.

\begin{lemma}\label{action_m}
For all $m \in \z$, we have:
$${\allowdisplaybreaks
\theta(m,0): \begin{cases}
u \longmapsto B^m u B^{-m}\\
v \longmapsto B^m v u^{-2m} B^{-m}\\
B \longmapsto B.
\end{cases}}$$
\end{lemma}

\begin{proof}
The result is clear if $m = 0$. If $m = 1$ then Proposition~\ref{section_p1} and Lemma~\ref{action_a1_a2} imply that:
\begin{align*}
\theta( 1 , 0)( u ) &=  \varphi(1,0) b_1 b_2 \varphi(1,0)^{-1} = b_1 a_1 b_2 a_2 b_1 b_2 a_2^{-1} b_2^{-1} a_1^{-1} b_1^{-1}
= b_1 a_1 B^{-1} b_2^{-1} b_1 a_1^{-1} b_1^{-1} \\
&= B b_1 b_2^{-1} b_1^{-1} b_2^{-2} b_1^{-1} = B u v u^{-1} v^{-1} u^{-1} = B u B^{-1},\\
\theta( 1 , 0 )( v ) &= \varphi(1,0) b_2^{-1} \varphi(1,0)^{-1} = b_1 a_1 b_2 a_2 b_2^{-1} a_2^{-1} b_2^{-1} a_1^{-1} b_1^{-1}
= b_1 a_1 b_2^{-1} B a_1^{-1} b_1^{-1} \\
&= b_1 b_2^{-1} b_1^{-1} B^{-1} = u v u^{-1} B^{-1} = B v u^{-2} B^{-1}, \;\text{and}\\
\theta( 1 , 0 )( B )
&=\varphi(1,0) B \varphi(1,0)^{-1} = b_1 a_1 b_2 a_2 B a_2^{-1} b_2^{-1} a_1^{-1} b_1^{-1}
= b_1 a_1 B^{-1} a_1^{-1} b_1^{-1}
= B.
\end{align*}
Suppose that the result holds for all integers belonging to $\{1, \ldots , m \}$, and let us show by induction that it holds for $m + 1$. We have:
\begin{align*}
\theta( m + 1 , 0)(u) &= \theta(m,0) (\theta(1,0)(u)) = \theta(m,0)(B u B^{-1}) = B^{m+1} u B^{-(m+1)},\\
\theta( m + 1 , 0)(v) &= \theta(m,0) (\theta(1,0)(v)) = \theta(m,0)( B v u^{-2} B^{-1}) = B^{m+1} v u^{ - 2(m+1) } B^{- (m+1)},\;\text{and}\\
\theta(m+1,0)(B) &= \theta(m,0)( \theta(1,0)(B)) = \theta(m,0)(B) = B,
\end{align*}
as required, and so the result holds for all $m\geq 0$. Now suppose that $m < 0$. Then $-m>0$, and so:
$$ B = \theta(0,0)(B) = \theta(m,0) (\theta(-m,0)(B)) = \theta(m,0) (B) .$$
We also have:
$$
u = \theta(m,0) (\theta(-m,0)(u)) = \theta(m,0) (B^{-m} u B^m) = B^{-m} \theta(m,0)(u) B^m,
$$
thus $\theta ( m , 0) = B^m u B^{-m}$, and:
$$
v = \theta(m,0) (\theta(-m,0)(v)) = \theta(m,0) ( B^{-m} v u^{2m} B^m)
= B^{-m} \theta(m,0)(v)  B^m u^{2m},
$$
hence $\theta ( m , 0 )( v) = B^m v u^{-2m} B^{-m}$ as required.
\end{proof}
		
\begin{lemma}\label{action_n}
For all $n \in \z$, we have:
\begin{multicols}{2}
${\allowdisplaybreaks
\theta(0,n): \begin{cases}
u \longmapsto B^{- \delta_n} u^{ (-1)^n } B^{ \delta_n} \\
v \longmapsto v B^{ \delta_n} \\
B \longmapsto B^{(-1)^n}
\end{cases}}$
where
${\allowdisplaybreaks
 \delta_n = \begin{cases}
0 & \text{if $n$ is even} \\
1 & \text{if $n$ is odd.}
\end{cases}}$
\end{multicols}
\end{lemma}

\begin{proof}
By relation~(v) of Theorem~\ref{presentation_p2_klein}, we have $\varphi ( 0, 1 ) = B^{-1} a_1^2 a_2 b_2^{-1}$. Using this equality, the idea of the proof is similar to that of Lemma~\ref{action_m}, and we leave the calculations to the reader.
\end{proof}

We are now able to describe the homomorphism $\theta$.\begin{theorem}\label{braid_group_klein}
The homomorphism $\theta:
\zsdz \to \operatorname{\text{Aut}} (F(u,v))$ is defined by:
\begin{equation*}
\theta(m,n): \begin{cases}
u \longmapsto B^{m - \delta_n} u^{(-1)^n} B^{-m + \delta_n} \\
v \longmapsto B^m v u^{-2m} B^{-m + \delta_n} \\
B \longmapsto B^{(-1)^n},
\end{cases} \;\text{where $B = [ u , v ]^\prime$ and}\; \delta_n=\begin{cases}
0 & \text{if $n$ is even} \\
1 & \text{if $n$ is odd.}
\end{cases}
\end{equation*}
With the notation of Remark~\ref{seq_split}, the isomorphisms $\xymatrix@1{ F(u,v) \rtimes_\theta ( \zsdz )
 \ar@<0.1cm>[r]^-\lambda & P_2(\klein) \ar@<0.1cm>[l]^-\gamma }$ are given by:
 \begin{multicols}{2}
 ${\allowdisplaybreaks
\lambda: \begin{cases}
(u;0,0) \longmapsto b_1 b_2 \\
(v;0,0) \longmapsto b_2^{-1} \\
(B;0,0) \longmapsto B \\
(\id ;1 , 0) \longmapsto b_1 a_1 b_2 a_2 \\
(\id; 0, 1) \longmapsto b_2 a_2
\end{cases}}$
 
${\allowdisplaybreaks
\gamma: \begin{cases}
a_1 \longmapsto ( v^{-1} u^{-1}; 1,1) \\
a_2 \longmapsto ( v; 0,-1) \\
b_1 \longmapsto ( u v ; 0,0) \\
b_2 \longmapsto ( v^{-1}; 0,0) \\
B \longmapsto (B;0,0).
\end{cases}}$
\end{multicols}
\noindent Up these isomorphisms and the identification given by Remark~\ref{identification_pi1_klein}, the homomorphism $(p_1)_\# : P_2 ( \klein ) \to \pi_1 ( \klein )$ is the projection onto the second coordinate, in other words, $(p_1)_\# ( w ; m , n) = ( m , n)$.
\end{theorem}

\begin{proof}
For all $(m , n) \in \zsdz$, we have $\theta( m , n) = \theta(m,0) \circ \theta(0,n)$. The first part of the statement follows by combining Lemmas~\ref{action_m} and~\ref{action_n}. Let us prove the second part. First,
$\lambda ( u ; 0 , 0) = u = b_1 b_2$, 
$\lambda ( v ; 0 , 0) = v = b_2^{-1}$, 
$\lambda( B ; 0 , 0) = B$, 
$\lambda( \id ; 1 , 0) = \id \varphi ( 1 , 0) = b_1 a_1 b_2 a_2$ and
$\lambda ( \id ; 0,1) = \id \varphi ( 0 , 1 ) = a_2^{-1} b_2^{-1}$. 
On the other hand:
\begin{align*}
\gamma ( a_1 ) &=  ( a_1 \varphi( \widetilde{a}_1 )^{-1} ; \widetilde{a}_1 ) = ( a_1 \left( \varphi( 1,0 ) \varphi(0,1) \right)^{-1} ; 1,1 )
= ( b_1^{-1} ; 1 , 1) = ( v^{-1} u^{-1} ; 1 ,1),\\
\gamma ( a_2 ) &= ( a_2 \varphi( \widetilde{a}_2)^{-1} ; \widetilde{a}_2 ) = ( a_2 \varphi (0,1) ; 0 , - 1)
= ( b_2^{-1} ; 0 , -1) = ( v ; 0 , -1), 		
\end{align*}
$\gamma ( b_1) = ( u v ; 0 , 0)$ and $\gamma ( b_2) = ( v^{-1} ; 0 , 0)$.
From the isomorphism $\lambda$ and Remark~\ref{identification_pi1_klein}, we see that $(p_1)_\# ( u ; 0 , 0) = (p_1)_\# ( v ; 0 , 0) = ( 0 , 0)$, $(p_1)_\# (\id ; 1 , 0) = (1,0)$ and $(p_1)_\# ( \id ; 0 , 1) = (0,1)$ as required.
\end{proof}

Let $c_\sigma : P_2 ( \klein ) \to P_2 ( \klein )$ be the automorphism of $P_{2}(T)$ given by conjugation by $\sigma$. By Theorem~\ref{presentation_p2_klein}, we have $c_\sigma( B ) = B$, $c_\sigma ( a_i ) = b_i$ and $c_\sigma ( b_j ) = B a_j B^{-1}$ for all $i,j \in \{ 1 , 2 \}$. Let $l_\sigma: F( u , v) \rtimes_\theta ( \zsdz ) \to F( u , v) \rtimes_\theta ( \zsdz )$ be defined by the composition $l_\sigma = \gamma \circ c_\sigma \circ \lambda$. To end this section, we give an explicit description of $l_\sigma$ on a set of generators.

\begin{theorem}\label{action_sigma_k2}
For all $r,s , m , n \in \z$, we have:
\begin{multicols}{2}
\begin{enumerate}[(1)]
\item $l_\sigma ( u^r ; 0 , 0) = ( (B u^{-1})^{-r} B^{-r} ; r , 0)$.
\item $l_\sigma ( v^s ; 0 , 0) = ( (u v)^{-s} ( u B)^{\delta_s} ; 0 , s )$.
\item $l_\sigma ( B ; 0 , 0) = ( B ; 0 , 0)$.
\item $l_\sigma ( \id ; m , 0) = ( \id ; m , 0 )$.
\item $l_\sigma ( \id ; 0,n) = ( B^{\delta_n} ; 0 , n)$. 
	
	\
	
\end{enumerate}
\end{multicols}
\end{theorem}
		
\begin{proof}
Equation~(3) follows easily from Theorem~\ref{braid_group_klein}. Let us prove equation~(1). For $r=0$, the result is clear. If $r=1$, we have:
\begin{align*}
l_\sigma(u ; 0,0) & = (\gamma \circ c_\sigma)( b_1 b_2) = \gamma( B a_1 a_2 B^{-1}) = (B;0,0) (v^{-1} u^{-1};1,1) (v;0,-1) (B^{-1},0,0) \\
& = (B v^{-1} u^{-1}  \theta(1,1)(v)  \theta(1,0)(B)^{-1}; 1,0) = ( B v^{-1} u^{-1} B v u^{-2} B^{-1}; 1,0) = (B u^{-1} B^{-1}; 1,0).
\end{align*}
Now suppose that the result holds for all integers belonging to $\{1,\ldots,r\}$, and let us show by induction that it is true for $r+1$. We have:
\begin{align*}
l_\sigma(u^{r+1};0,0) & = l_\sigma(u;0,0)  l_\sigma(u^r;0,0) = ( B u^{-1} B^{-1}; 1,0 ) ((Bu^{-1})^r B^{-r};r,0) \\
& = ( B u^{-1} B^{-1} \left( \theta(1,0)(B)  \theta(1,0)(u)^{-1} \right)^r \theta(1,0)(B)^{-r}; r+1,0) \\
& = (B u^{-1} B^{-1} \left( B (BuB^{-1})^{-1} \right)^r B^{-r}; r+1,0) = ( (Bu^{-1})^{r+1} B^{-(r+1)} ; r,0).
\end{align*}
and hence equation~(1) holds for all $r\geq 0$. Now assume that $r < 0 $. Then $-r > 0 $, and since $l_\sigma(u;0,0) = l_\sigma(u^{-r};0,0)^{-1}$, it follows that:
\begin{align*}
l_\sigma(u^r;0,0) & = l_\sigma(u^{-r};0,0)^{-1} = ((B u^{-1})^{-r} B^r ; -r,0)^{-1}
= ( ( \theta(r,0)( (Bu^{-1})^{-r} B^r ) )^{-1} ; r,0 ) \\
& = ( ( ( \theta(r,0)(B)  \theta(r,0)(u)^{-1} )^{-r}  \theta(r,0)(B)^r )^{-1};r,0 ) \\
& =  ( ( ( B  (B^r u B^{-r})^{-1} )^{-r}  B^r )^{-1} ; r,0 ) 
= ( (B u^{-1})^{r} B^{-r};r,0 ).
\end{align*}
This proves equation~(1). Let us show that equation~(2) holds. If $s = 0$, the result is immediate. For $s=1,2$, we have:
\begin{align}
l_\sigma(v;0,0) & =  (\gamma \circ c_\sigma)( b_2^{-1}) = \gamma( B a_2^{-1} B^{-1})  = (B;0,0) (v;0,-1)^{-1} (B;0,0)^{-1} \nonumber \\
& =  (B;0,0) (\theta(0,1)(v)^{-1};0,1) (B^{-1};0,0)
= (B (vB)^{-1} \theta(0,1)(B)^{-1}; 0,1) \nonumber \\
& = (v^{-1} B; 0,1) = ( (u v)^{-1} (u B)^{\delta_1}; 0,1), \;\text{and}\label{l_sigma_1}\\
l_\sigma(v^2;0,0) & = l_\sigma(v;0,0) l_\sigma(v;0,0) = (v^{-1} B;0,1) (v^{-1} B;0,1) 
= (v^{-1} B  \theta(0,1)(v)^{-1}  \theta(0,1)(B) ; 0,2) \nonumber \\
& = ((v^{-1} B (vB)^{-1} B^{-1}; 0,2) = (v^{-2} B^{-1} ; 0,2)  = ( (uv)^{-2} (u B)^{\delta_2}; 0,2),\label{l_sigma_2}
\end{align}
and so equation~(2) holds in these two cases. Now suppose that $s$ is even. Then $s = 2k$ for some $k \in \z$. We wish to show that $l_\sigma ( v^{2k} ; 0 , 0) = ( (uv)^{-2k} ; 0 , 2k)$. If $k = 1$, the result holds by~(\ref{l_sigma_2}). Suppose that $k \geq 1$ and that the result is true for all integers belonging to $\{ 1, \ldots ,k \}$. Then:
\begin{align*}
l_\sigma(v^{2(k+1)};0,0) & = l_\sigma(v^{2k};0,0)  l_\sigma(v^{2};0,0)
= ((uv)^{-2k};0,2k)  ( (uv)^{-2}; 0,2) \\
& =  ((uv)^{-2k}  ( \theta(0,2k)(u)  \theta(0,2k)(v) )^{-2};0,2k + 2) \\
& = ( (uv)^{-2k} (uv)^{-2};0,2(k+1))= ( (uv)^{-2(k+1)} ;0,2(k+1)).
		\end{align*}
and so by induction, equation~(2) holds for all $k\geq 1$. Now suppose that $k < 0$. Then $-k > 0$, and since $l_\sigma(v^{2k},0,0) = l_\sigma(v^{-2k};0,0)^{-1}$, we see that:
\begin{align*}
l_\sigma(v^{2k} ; 0,0) & = l_\sigma(v^{-2k};0,0)^{-1}  = ((uv)^{2k};0,-2k)^{-1} = ( ( \theta(0,2k)((uv)^{2k}))^{-1} ;0,2k  ) \\
& = ( ( (\theta(0,2k)(u)  \theta(0,2k)(v))^{2k})^{-1} ; 0,2k) = ((uv)^{-2k};0,2k),
\end{align*}
and hence equation~(2) holds for all $s$ even. Now assume that $s$ is odd, and let $k \in \z$ be such that $s = 2k+1$. Using~(\ref{l_sigma_1}) and the result for $s$ even, we obtain:
\begin{align*}
l_\sigma(v^{2k+1};0,0) & = l_\sigma(v^{2k};0,0)  l_\sigma(v;0,0) = ((uv)^{-2k};0,2k) (v^{-1} B;0,1) \\
& = ( (uv)^{-2k}  \theta(0,2k)(v)^{-1}  \theta(0,2k)(B); 0, 2k + 1) \\
& = ( (uv)^{-2k} v^{-1} B; 0, s) = ( (uv)^{-s} (u B)^{\delta_s}, 0, s).
\end{align*}
This completes the proof of equation~(2). To obtain equation~(4), it suffices to prove the result for $m = 1$. We have:
\begin{align*}
l_\sigma(\id;1,0) & = (\gamma \circ c_\sigma)( b_1 a_1 b_2 a_2) = \gamma( B a_1 B^{-1} b_1 B a_2 B^{-1} b_2) \\
& = (B;0,0) (v^{-1} u^{-1};1,1) (B;0,0)^{-1} (u v ;0,0) (B;0,0) (v;0,-1) (B;0,0)^{-1} (v^{-1};0,0) \\
& = (B v^{-1} u^{-1} \theta(1,1)(B)^{-1}  \theta(1,1)(u)  \theta(1,1)(v) \theta(1,1)(B)  \theta(1,1)(v)  \theta(1,0)(B)^{-1} 
\theta(1,0)(v)^{-1} ; 1,0) \\
& = (B v^{-1} u^{-1} B u^{-1} B v u^{-2} B^{-1} B v u^{-2} B^{-1} (B v u^{-2} B^{-1})^{-1} ; 1,0) = (\id ; 1,0),
\end{align*}
as required. Finally, to prove equation~(5), if $n \in \z$, let $k \in \z$ be such that $n = 2k + \delta_n$. Then:
\begin{align*}
 l_\sigma(\id;0,n) & = l_\sigma(\id;0,1)^n = (B;0,1)^n = (B;0,1)^{2k} (B;0,1)^{\delta_n} \\
& = ((B;0,1) (B;0,1))^k (B^{\delta_n};0,\delta_n)= (B  \theta(0,1)(B) ; 0,2)^k (B^{\delta_n},0,\delta_n) \\
& = (\id;0,2k) (B^{\delta_n};0,\delta_n)
 = (\theta(0,2n)(B)^{\delta_n}; 0, 2k + \delta_n) = (B^{\delta_n}; 0,n),
\end{align*}
and thus equation~(5) holds.
\end{proof}

\section{Proof of Theorems~\ref{BORSUK_TAU_1} and~\ref{BORSUK_TAU_2}}\label{borsuk_maps_torus}
		
The purpose of this section is to prove Theorems~\ref{BORSUK_TAU_1} and~\ref{BORSUK_TAU_2}. We identify $\pi_1 ( \torus )$ with $\zsz$ and $\pi_1 ( \klein )$ with $\zsdz$ as in Sections~\ref{braid_torus} and~\ref{braid_klein} respectively. Consider the following short exact sequences: 
\begin{equation}\label{homo_tau_1}
\xymatrix{
1 \ar[r] & \pi_1 ( \torus ) \ar[r]^-{i_1} & \pi_1 ( \torus ) \ar[r]^-{\theta_1} & \ztwo \ar[r] & 1, 	
}\; \text{where:}\end{equation}
	
\begin{center}
\begin{multicols}{2}
${\allowdisplaybreaks
i_1: \begin{cases}
(1,0) \longmapsto (2,0) \\
(0,1) \longmapsto (0,1)
\end{cases}}$		
		
${\allowdisplaybreaks
\theta_1: \begin{cases}
(1,0) \longmapsto \overline{1} \\
(0,1) \longmapsto \overline{0}
\end{cases}}$
\end{multicols}
\end{center}
and
\begin{equation}\label{homo_tau_2}
\xymatrix{
1 \ar[r]  & \pi_1 ( \torus ) \ar[r]^-{i_2} & \pi_1 ( \klein ) \ar[r]^-{\theta_2} & \ztwo \ar[r] & 1, 	
}\; \text{where:}\end{equation}
		
\begin{center}
\begin{multicols}{2}
${\allowdisplaybreaks
i_2: \begin{cases}
(1,0) \longmapsto (1,0) \\
(0,1) \longmapsto (0,2)
\end{cases}}$

${\allowdisplaybreaks
\theta_2: \begin{cases}
(1,0) \longmapsto \overline{0} \\
(0,1) \longmapsto \overline{1}.
\end{cases}}$
\end{multicols}
\end{center}
By covering space theory, there exist double coverings $c_1: \torus \to \torus$ and $c_2 : \torus \to \klein$ such that the induced homomorphisms on the fundamental group are $i_1$ and $i_2$ respectively. If $k \in \{ 1 , 2 \}$ and $\tau_k: \torus \to \torus$ is the non-identity deck transformation associated with $c_k$, then $\tau_k$ is a free involution, and the short exact sequence induced by $\tau_k$ is (\ref{homo_tau_1}) if $k=1$ and is (\ref{homo_tau_2}) if $k=2$. Further, $\tau_{k}$ lifts to a homeomorphism $\widehat{\tau}_{k}: \rtwo \to \rtwo$, where $\widehat{\tau}_{1}(x,y)=(x+\frac{1}{2},y)$ and $\widehat{\tau}_{2}(x,y)=(x+\frac{1}{2},1-y)$ for all $(x,y)\in \rtwo$. As we shall now see, up to the equivalence relation defined just after Proposition~\ref{reduce_involution}, these are the only free involutions of the torus. 

\begin{proposition}\label{class_involutions_torus}
The free involutions $\tau_1$ and $\tau_2$ are not equivalent. Further,  let $\tau: \torus \to \torus$ be a free involution. If $\tau$ preserves (resp.\ reverses) orientation, then $\tau$ is equivalent to $\tau_1$ (resp.\ to $\tau_2$).
\end{proposition}
		
\begin{proof}
The short exact sequences~(\ref{homo_tau_1}) and~(\ref{homo_tau_2}) are not equivalent (in the sense described at the end of Section~\ref{generalities}), since the middle groups are not isomorphic, and so $\tau_1$ and $\tau_2$ are not equivalent by Proposition~\ref{invo_seq_exact}. Now let $\tau: \torus \to \torus$ be a free involution. Recall that $\torus_\tau$ is the corresponding orbit space, and that the natural projection $p_\tau: \torus \to \torus_\tau$ is a double covering. Suppose first that $\tau$ preserves orientation. Then $\torus_\tau$ is homeomorphic to the torus, and the short exact sequence induced by $\tau$ is of the following form:
\begin{equation}\label{seq_aux_1}\xymatrix{
1 \ar[r] & \pi_1 ( \torus ) \ar[r]^-{(p_{\tau})_\#} & {\underbrace{\pi_1 ( \torus_\tau )}_{\cong \zsz}}  \ar[r]^-{\theta_\tau} & \ztwo \ar[r] & 1.
}\end{equation}
By~\cite[Proposition~30]{GonGua}, there exists an isomorphism $\psi: \pi_1 ( \torus ) \to \pi_1 ( \torus_\tau)$ such that $\theta_{\tau} \circ \psi = \theta_1$. Comparing diagrams~(\ref{homo_tau_1}) and~(\ref{seq_aux_1}) and using exactness and the 5-Lemma, it follows that $\psi$ induces an isomorphism $\varphi: \pi_1 ( \torus ) \to \pi_1 ( \torus )$ such that $(p_\tau)_\# \circ \varphi = \psi \circ i_1$. So the exact sequences~(\ref{homo_tau_1}) and~(\ref{seq_aux_1}) are equivalent, and we conclude that $\tau$ and $\tau_1$ are equivalent by Proposition~\ref{invo_seq_exact}. Now suppose that $\tau$ reverses orientation. Then $\torus_\tau$ is homeomorphic to the Klein bottle, and the short exact sequence induced by $\tau$ is of the following form:
\begin{equation}\label{seq_aux_2}\xymatrix{
1 \ar[r] & \pi_1 ( \torus ) \ar[r]^-{(p_{\tau})_\#} & {\underbrace{\pi_1 ( \torus_\tau )}_{\cong \zsdz}} \ar[r]^-{\theta_\tau} & \ztwo \ar[r] & 1.
}\end{equation}
Using the identification described in Remark~\ref{identification_pi1_klein}, we have an isomorphism 
$\gamma:\zsdz \to  \pi_1 ( \torus_\tau )$. The elements $\gamma(0,1) = v_1$ and $\gamma(1,1) = u_1 v_1$ represent loops in $\torus_\tau$ that reverse orientation. Since $p_\tau: \torus \to \torus_\tau$ is a covering, we have $\theta_\tau\circ \gamma (1,0) = \theta_\tau\circ \gamma(1,1) = \overline{1}$, so the element $(1,0) = (1,1)( 0,1 )^{-1} $ satisfies $\theta_\tau \circ \gamma( 1,0 ) = \overline{0}$, and thus $\theta_\tau\circ \gamma = \theta_2$. Let $\psi: \pi_1 ( \klein ) \to \pi_1 ( \torus_\tau)$ be an isomorphism. Comparing diagrams~(\ref{homo_tau_2}) and~(\ref{seq_aux_2}) and using exactness and the 5-Lemma once more, it follows that $\psi$ induces an isomorphism $\varphi: \pi_1 ( \torus ) \to \pi_1 ( \torus )$ such that $(p_\tau)_\# \circ \varphi = \psi \circ i_2$. So the exact sequences~(\ref{homo_tau_2}) and~(\ref{seq_aux_2}) are equivalent, and we conclude that $\tau$ and $\tau_2$ are equivalent by Proposition~\ref{invo_seq_exact}.
\end{proof}

Theorem~\ref{set_homotopy} gives rise to the following commutative diagram, where the maps are bijections (we simplify the notation):
\begin{equation}\begin{gathered}\label{set_homo_torus}\xymatrix{
[ \torus , * ; \torus , * ] \ar@/^0.7cm/[rr]^{\Lambda} \ar[r]^-{\Gamma} & \hom ( \zsz , \zsz ) & [\torus , \torus] \ar[l]_-{\Delta}.
}\end{gathered}\end{equation}
We shall now solve the Borsuk-Ulam problem for homotopy classes with respect to the involution $\tau_1$. We first state and prove the following lemma.

\begin{lemma}\label{reduce_diagram_tau1}
A free homotopy class $\beta \in [ \torus , \torus ]$ does not have the Borsuk-Ulam property with respect to $\tau_1$ if and only if there exist pure braids $a,b \in P_2 ( \torus )$ such that:
\begin{enumerate}[(i)]
\item $a l_\sigma ( b) = b a$.
\item $\beta_\# (1,0) = (p_1)_\# ( a l_\sigma (a) )$.
\item $\beta_\# (0,1) = (p_1)_\# ( b )$.
\end{enumerate}
\end{lemma}

\begin{proof}
Suppose that $\beta$ does not have the Borsuk-Ulam property with respect to  $\tau_1$. By~(\ref{set_homo_torus}) and Proposition~\ref{borsuk_pointed_free}, there exists a pointed homotopy class $\alpha$ that does not have the Borsuk-Ulam property and for which $\alpha_\free = \Lambda ( \alpha ) =\beta$. Using Proposition~\ref{borsuk_braid}, there exist homomorphisms $\varphi: \pi_1 ( \torus ) \to P_2 ( \torus )$ and $\psi : \pi_1 ( \torus ) \to B_2 ( \torus )$ that make diagram~(\ref{diag_borsuk_braid}) commute. By~(\ref{homo_tau_1}), $\psi(1,0)\in B_2 ( \torus ) \backslash P_2 ( \torus )$ and $\psi(0,1)\in P_2 ( \torus )$. From the short exact sequence~(\ref{seq_braid}) and Theorem~\ref{remark_p2_torus}, there exist $a,b \in P_2 ( \torus)$ such that: 
\begin{enumerate}[(1)]
\item $\psi(1,0) = a \sigma$.
\item $\psi(0,1) = b$.
\end{enumerate}
Since $\pi_1 (\torus)$ is an Abelian group, we have $[\psi(1,0),\psi(0,1)] = a \sigma b \sigma a^{-1} a^{-1} b^{-1} = \id$, which is equivalent to $a l_\sigma ( b) = b a$, from which we obtain~(i). The commutativity of~(\ref{diag_borsuk_braid}) implies that: 
\begin{enumerate}[(1)]\setcounter{enumi}{2}
\item $\varphi(1,0) = \psi ( (p_{\tau_1})_\# (1,0) ) = \psi(2,0) = (a \sigma)^2 = a l_\sigma ( a) \sigma^2$, and
\item $\varphi(0,1) = \psi ( (p_{\tau_1})_\# (0,1) ) = \psi(0,1) = b$.
\end{enumerate}
Finally, using (\ref{set_homo_torus}), we have:
\begin{align*}
\beta_\# (1,0) &= (p_1)_\# ( \varphi(1,0)) = (p_1)_\# ( a l_\sigma ( a) ) + (p_1)_\# ( B , 0 , 0) = (p_1)_\# ( a l_\sigma ( a) ), \;\text{and}\\  
\beta_\# (0,1) &= (p_1)_\# ( \varphi(0,1)) = (p_1)_\# ( b ),
\end{align*}
from which we obtain~(ii) and~(iii). Conversely, suppose that there exist $a, b \in P_2 ( \torus )$ such that~(i),~(ii) and~(iii) are satisfied. We define $\psi : \pi_1 ( \torus ) \to P_2 ( \torus )$ (resp.\ $\varphi: \pi_1 ( \torus ) \to B_2 ( \torus)$) on the generators of $\pi_1 ( \torus )$ by the correspondences~(1) and~(2) (resp.\ (3) and~(4)). Since:
\begin{align*}
\left[ \varphi(1,0) , \varphi(0,1) \right]
& = \left[ a l_\sigma (a) \sigma^2 , b \right]
= a \sigma a \sigma^{-1} \sigma^2 b \sigma^{-2} \sigma a^{-1} \sigma^{-1} a^{-1} b^{-1} \\
& = a l_\sigma ( a l_\sigma (b) a^{-1} ) a^{-1} b^{-1}
\stackrel{(i)}{=}  a l_\sigma( b a a^{-1}) a^{-1} b^{-1}
\stackrel{(i)}{=}  b a a^{-1} b^{-1}  =  \id, \;\text{and}\\
\left[ \psi(1,0) , \psi(0,1) \right] &=  \left[ a \sigma , b \right] = a \sigma b \sigma^{-1} a^{-1} b^{-1}
= a l_\sigma ( b) a^{-1} b^{-1} \stackrel{(i)}{=}  b a a^{-1} b^{-1} = \id .		
\end{align*}
these correspondences extend to homomorphisms. As in the first part of the proof, $\psi$ and $\varphi$ make the diagram~(\ref{diag_borsuk_braid}) commute, so by Theorem~\ref{borsuk_braid} and~(\ref{set_homo_torus}), $\beta$ does not have the Borsuk-Ulam property with respect to  $\tau_1$.
\end{proof}

For each $\beta \in [\torus, \torus]$, let $\beta_{1,1} , \beta_{2,1} , \beta_{1,2} , \beta_{2,2}\in \z$ be such that: 
$$\beta_\# (1,0) = \Delta ( \beta )(1,0) = ( \beta_{1,1} , \beta_{2,1} ) \quad \text{ and } \quad \beta_\# (0,1) = \Delta ( \beta )(0,1) = ( \beta_{1,2} , \beta_{2,2} ).$$

\begin{proposition}\label{classification_tau1}
If $\beta \in [\torus , \torus]$ is a homotopy class then $\beta$ does not have the Borsuk-Ulam property with respect to $\tau_1$.
\end{proposition}

\begin{proof}
Let $r,s \in \z$ and $i,j \in \{ 0 , 1 \}$ be such that $\beta_{1,1} = 2r + i$ and 
$\beta_{2,1} = 2 s + j$. With the identification of~Theorem \ref{remark_p2_torus},  let $a = ( x^i y^j , r , s )$ and $b = ( \id , \beta_{1,2} , \beta_{2,2})$ of $P_2 ( \torus )$. Let us show that these elements satisfy~(i),~(ii) and~(iii) of Lemma~\ref{reduce_diagram_tau1}, from which it will follow that $\beta$ does not have the Borsuk-Ulam property with respect to  $\tau_1$. First:
$$a l_\sigma ( b)
= ( x^i y^j , r , s ) ( \id , \beta_{1,2} , \beta_{2,2})
= ( \id , \beta_{1,2} , \beta_{2,2}) ( x^i y^j , r , s )
= ba,$$
which yields~(i). Using~(\ref{formula_abelian}), we obtain:
$$(p_1)_\# ( a l_\sigma (a ) )
= (r,s) + (i, j) + (r,s) = ( \beta_{1,1} , \beta_{2,1}) = \beta_\# (1,0),$$
which proves~(ii). Finally, we have:
$$(p_1)_\# ( b )
= ( \beta_{1,2} , \beta_{2,2})
= \beta_\# (0,1),$$
whence~(iii).		
\end{proof}

\begin{proof}[Proof of Theorem \ref{BORSUK_TAU_1}]
The result is a consequence of Propositions~\ref{reduce_involution},~\ref{class_involutions_torus} and~\ref{classification_tau1}.
\end{proof}

We now solve the Borsuk-Ulam problem with respect to the involution $\tau_2$. This case is more delicate than the previous one. We first prove some preliminary results.

\begin{lemma}\label{reduce_diagram_tau_2}
A free homotopy class $\beta \in [ \torus , \torus ]$ does not have the Borsuk-Ulam property with respect to  $\tau_2$ if and only if there exist $a,b \in P_2 ( \torus )$ such that:
\begin{enumerate}[(i)]
\item $a b l_\sigma ( a ) = b$.
\item $\beta_\# (1,0) = (p_1)_\# ( a )$.
\item $\beta_\# (0,1) = (p_1)_\# ( b l_\sigma ( b ) )$.
\end{enumerate}
\end{lemma}
		
\begin{proof}
We follow the proof of Lemma~\ref{reduce_diagram_tau1}, replacing the homomorphism $\psi$ (resp.\ $\varphi$) by the homomorphism $\psi: \pi_1 ( \klein ) \to B_2 ( \torus )$ (resp.\ $\varphi: \pi_1 ( \torus ) \to P_2 ( \torus )$) defined on the generators of $\pi_1 ( \klein )$ (resp.\ of $\pi_1 ( \torus )$) by $\psi(1,0) = a$ and $\psi(0,1) = b \sigma$  (resp.\ $\varphi(1,0) = a$ and $\varphi(0,1) = b l_\sigma ( b) \sigma^2$). The details are left as an exercise.
\end{proof}

\begin{proposition}\label{classification_tau2_first_part}
Suppose that $\beta \in [ \torus , \torus ]$ satisfies one of the following conditions:
\begin{enumerate}
\item at least one of $\beta_{1,2}$ and $\beta_{2,2}$ is odd.
\item $( \beta_{1,1} , \beta_{2,1} ) = ( 0 , 0)$ and $\beta_{1,2}$ and $\beta_{2,2}$ are both even.
\end{enumerate}
Then $\beta$ does not have the Borsuk-Ulam property with respect to $\tau_2$.		
\end{proposition}
		
\begin{proof}
Assume that condition~1 of the statement is satisfied. Suppose first that $\beta_{1,2}$ and $\beta_{2,2}$ are odd. Let $r , s \in \z$ be such that $\beta_{1,2} = 2 r + 1$ and $\beta_{2,2} = 2 s + 1$. With the identification of Theorem~\ref{remark_p2_torus}, consider the  elements $a = ( x^{- 2 \beta_{1,1}} y^{-2 \beta_{2,1}} , \beta_{1,1} , \beta_{1,2})$ and $b = ( y^{1 + 2 \beta_{2,1}} x^{-1} , r + 1 , s - \beta_{2,1} )$ of $P_2 ( \torus )$. We will show that $a$ and $b$ satisfy~(i),~(ii) and~(iii) of Lemma~\ref{reduce_diagram_tau_2}, from which it will follow that $\beta$ does not have the Borsuk-Ulam property with respect to $\tau_2$. Using~(\ref{formula_abelian}) we obtain:
\begin{align*}
a b l_\sigma ( a) = & ( x^{- 2 \beta_{1,1}} y^{-2 \beta_{2,1}} , \beta_{1,1} , \beta_{1,2}) 
( y^{1 + 2 \beta_{2,1}} x^{-1} , r + 1 , s - \beta_{2,1} )
l_\sigma ( x^{- 2 \beta_{1,1}} y^{-2 \beta_{2,1}} , \beta_{1,1} , \beta_{1,2}) \\
= & ( x^{-2 \beta_{1,1}} y x^{-1}, \beta_{1,1} + r + 1, s) (xy^{-1} x^{2 \beta_{1,1}} y^{2 \beta_{1,2}} y x^{-1}, - \beta_{1,1} , - \beta_{1,2}) \\
= & ( y^{1 + 2 \beta_{2,1}} x^{-1} , r + 1 , s - \beta_{2,1} ) = b,
		\end{align*}
and hence condition~(i) is satisfied. Further:
\begin{align*}
(p_1)_\# ( a ) &= ( \beta_{1,1} , \beta_{1,2}) \;\text{and}\\
(p_1)_\# ( b l_\sigma (b)) &= (r + 1 , s - \beta_{2,1}) + ( -1 , 1 + 2 \beta_{2,1} ) + ( r + 1 , s - \beta_{2,1} ) = ( \beta_{1,2}, \beta_{2,2}),
\end{align*}
from which it follows that conditions~(ii) and~(iii) are satisfied, and hence $\beta$ does not have the Borsuk-Ulam property with respect to $\tau_2$. The remaining cases are analogous, and we just provide the elements $a$ and $b$ of $P_2 ( \torus)$ in each case that satisfy~(i),~(ii) and~(iii) of Lemma~\ref{reduce_diagram_tau_2}, the details being left as an exercise. If $\beta_{1,2}$ is odd and $\beta_{2,2}$ is even, we take $\beta_{1,2} = 2 r + 1$ and $\beta_{2,2} = 2s$, where $r,s \in \z$, and $a = (x^{- \beta_{1,1} } y^{-2 \beta_{2,1} - 1} x^{- \beta_{1,1}} y, \beta_{1,1}, \beta_{2,1} )$ and $b = (  x^{-1} , r + 1 , s )$. If $\beta_{1,2}$ is even and $\beta_{2,2}$ is odd, we set $\beta_{1,2} = 2 r$ and $\beta_{2,2} = 2s + 1$, where $r,s \in \z$, and $a = ( x^{- 2 \beta_{1,1} + 1} y^{- \beta_{2,1} } x^{-1} y^{- \beta_{2,1} } , \beta_{1,1} , \beta_{2,1} )$ and $b = ( x^{-2 \beta_{1,1} + 1} y x^{-1}, \beta_{1,1} + r  , s )$. Finally, if $\beta$ satisfies condition~2 of the statement, we take $\beta_{1,2} = 2 r $ and $\beta_{2,2} = 2s$, where $r,s \in \z$, and $a = ( \id , 0 , 0)$ and $b = (\id, r , s)$. 
\end{proof}
		
Note that the case where $( \beta_{1,1} , \beta_{2,1} ) \neq ( 0 , 0)$ and $\beta_{1,2}$ and $\beta_{2,2}$ are even is not addressed by  Proposition~\ref{classification_tau2_first_part}. Before analysing this situation, let $z = z(x,y)$ be a reduced word in $F(x,y)$. Then $z$ is called a \emph{palindrome} if $z(x,y) =  z(x^{-1},y^{-1})^{-1}$. This means that $z$ reads the same backwards as forwards. If $z=w_{1}^{\epsilon_{1}}\cdots w_{m}^{\epsilon_{m}}$, where for all $i=1,\ldots,m$, $w_{i}\in \{ x,y\}$ and $\epsilon_{i}\in \mathbb{Z}$, the \emph{length} $\ell(z)$ of $z$ is defined to be equal to $\sum_{i=1}^{m} \lvert \epsilon_{i} \rvert$. We now prove two technical lemmas.

\begin{lemma}\label{palin_1}
If $z = z(x,y) \in F( x,y )$ is a palindrome, then $| z |_x$ or $| z |_y$ is even.		
\end{lemma}
		
\begin{proof}
Let $z=z(x,y)$ be a reduced word in $F(x,y)$. If $\ell(z)=0$ or $\ell ( z ) = 1$ then $z$ is trivial or $z \in \{ x , x^{-1} , y , y^{-1} \}$ respectively, and at least one of $| z |_x$ and $| z |_y$ is zero. If $\ell ( z ) =2$ then $z\in \{ x^{2a}, y^{2a}, x^a y^b,y^b x^a \ | \ a,b \in \{ -1,1 \} \}$. Since $z$ is a palindrome, $z = x^{2a}$ or $z = y^{2b}$ and the result follows. So suppose by induction that the result holds for palindromes of length less than or equal to $r$, where $r \geq 2$, and let $z$ be a palindrome such that $\ell(z)=r+1$.  By definition of palindrome, we have $z = k z' k$, where $ k \in \{ x , x^{-1} , y , y^{-1} \}$, $z' \in F(x,y)$ and $k z'k$ is in reduced form. Then $\ell(z')=r - 2$ and
$$k z' k
= z = z(x^{-1},y^{-1})^{-1}
= (k^{-1} z'( x^{-1} , y^{-1}) k^{-1})^{-1}
= k z' (x^{-1},y^{-1})^{-1} k.$$
Thus $z'$ is a palindrome, and $| z' |_x$ or $| z' |_y$ is even by induction. The result follows using the fact that $| z |_g \equiv | z ' |_g \ {\rm mod} \ 2$ for all $g \in \{ x , y \}$.
\end{proof}

\begin{lemma}\label{palin_2}
Let $z,w \in F(x,y)$ be such that $w$ is non trivial, $| w |_x$ and $| w |_y$ are even and
\begin{equation}\label{eq_palin}
z w z^{-1} = w( x^{-1} , y^{-1} )^{-1}.	
\end{equation}
Then $| z |_x$ or $| z |_y$ is even.
\end{lemma}
		
\begin{proof}
The result is immediate if $z$ is trivial. So assume that $z$ is non trivial. Without loss of generality, we may suppose that $z$ and $w$ are each written in reduced form. Since $\ell (w)= \ell (w( x^{-1} , y^{-1} )^{-1})$, there must be cancellation on the left-hand side of equation~(\ref{eq_palin}). We are interested in the parity of $| z |_x$ and $| z |_y$. The following operations show that, while respecting the hypotheses of the statement of the lemma, we may modify $z$ and $w$ so as to control better the possible cancellations in~(\ref{eq_palin}). In what follows, let $g \in \{ x, y \}$.

\

\noindent {\it Step 1: We may assume that cancellation occurs at $zw$ or at $w z^{-1}$, but not at both.} For suppose that cancellation occurs at $z w$ and at $w z^{-1}$. Then $z = z_1 k$ and $w = k^{-1} w_1 k$ in reduced form for some $k \in \{ x , x^{-1} , y , y^{-1} \}$ and some $z_1 , w_1 \in F(x,y)$. Now $w_1$ is non trivial and $| w_1 |_x$ and $| w_1 |_y$ are even. Also, $(k^{-1} z_1) w_1 (k^{-1} z_1)^{-1} = w_1 (x^{-1},y^{-1})^{-1}$ by (\ref{eq_palin}). This equation is of the same nature as that of $(\ref{eq_palin})$, and it is easy to see that $| k^{-1} z_1|_g \equiv | z |_g \mod 2$. If cancellation occurs at $(k^{-1} z_1) w_1$ and at $w_1 (k^{-1} z_1)^{-1}$, we repeat the argument. After a finite number of operations (possibly zero, and at most $\ell(z)$ times), we obtain elements $z', w' \in F(x,y)$ in reduced form for which $| z' |_g \equiv | z |_g \mod 2$, that continue to satisfy the hypotheses of the lemma, and such that any cancellation in $z'w' z'^{-1}$ occurs either at $z w'$ or at $z w'^{-1}$, but not at both. 

\

\noindent {\it Step 2: We may assume that cancellation occurs at $zw$.} For if cancellation occurs at $w z^{-1}$, by setting $w' = w^{-1}$, the elements $z , w' \in F(x,y)$ satisfy the hypotheses of the lemma, and cancellation occurs at $z w'$.

\

\noindent {\it Step 3: We may assume that $z$ cannot be written in reduced form $z' w^{\varepsilon}$, where $\varepsilon \in \{-1,1 \}$.} For if $z=z_1 w^{\varepsilon_1}$ in reduced form for some $z_1 \in F(x,y)$ and some $\varepsilon_1 \in \{ -1 , 1 \}$, then $z_1 w z_1^{-1} = w( x^{-1} , y^{-1})^{-1}$ by~(\ref{eq_palin}). Hence the elements $z_1 , w \in F(x,y)$ satisfy the hypotheses of the lemma, and $| z |_g = | z_1 |_g + \varepsilon_1 | w |_g \equiv | z_1 |_g \mod 2 $. If $z_1$ may be writen in reduced form $z_2 w^{\varepsilon_2}$, for some $z_2 \in F(x,y)$ and some $\varepsilon_2 \in \{ -1 , 1 \}$, we repeat the argument. After a finite number of operations, we obtain an element $z' \in F(x,y)$ such that $z'$ and $w$ satisfy the hypotheses of the lemma, $z'$ cannot be written in reduced form $\tilde{z} w^{\varepsilon}$ for some $\tilde{z} \in F(x,y)$ and some $\varepsilon \in \{ - 1 , 1 \}$, and $|z|_g \equiv | z'|_g \mod 2$. 

\

After applying steps 1, 2 and 3 as many times as necessary, we may suppose without loss of generality that there exist elements $z_1, z_2 , w_1 , w_2 \in F(x,y)$ such that $z = z_1 z_2$ and $w = w_1 w_2$ in reduced form, $w_1$ and $w_2$ are non trivial, $w_1 = z_2^{-1}$, and there is no cancellation either at $z_1 w_2$ or at $w_2 z_2^{-1}$. By~(\ref{eq_palin}), we have:
\begin{equation}\label{eq_palin_2}
z_1 w_2 w_1 z_1^{-1} = w_2 (x^{-1}, y^{-1})^{-1} w_1 (x^{-1}, y^{-1})^{-1}.
\end{equation}
By hypothesis, there is no cancellation in $z_1 w_2 w_1 z_1^{-1}$. But $\ell (w_2 w_1)= \ell (w_1 w_2)$, which in turn is equal to $\ell (w_2 (x^{-1}, y^{-1})^{-1} w_1 (x^{-1}, y^{-1})^{-1})$. So by~(\ref{eq_palin_2}), $z_1$ is trivial and $w_i = w_i (x^{-1}, y^{-1})^{-1}$, for all $i \in \{ 1,2 \}$. Thus $z = w_1^{-1}$, and therefore it is a palindrome. The result then follows from Lemma~\ref{palin_1}. 
\end{proof}
		
\begin{proposition}\label{classification_tau2_second_part}
Suppose that $\beta \in [ \torus , \torus ]$ satisfies $( \beta_{1,1} , \beta_{2,1} ) \neq ( 0 , 0)$, and $\beta_{1,2}$ and $\beta_{2,2}$ are even. Then $\beta$ has the Borsuk-Ulam property with respect to $\tau_2$.
\end{proposition}
		
\begin{proof}
We argue by contradiction. Suppose that $\beta$ does not have the Borsuk-Ulam property with respect to  $\tau_2$. By~Lemma \ref{reduce_diagram_tau_2}, there exist $ a = (w(x,y) , r_a , s_a )$ and $b = ( z(x,y) , r_b , s_b )$ such that $a$, $b$ and $\beta_\# $ satisfy items~(i),~(ii) and~(iii). By~(ii), we have $( r_a , s_a ) = ( \beta_{1,1} , \beta_{2,1})$. Using~(iii),~(\ref{formula_abelian}) implies that: 
$$( \beta_{1,2} , \beta_{2,2} )
=  (p_1)_\#(b l_\sigma(b))
=  (|z|_x + 2r_b , |z|_y + 2s_b),$$
and we see that $|z|_x$ and $|z|_y$ are even. Further:
\begin{align*}
a b l_\sigma(a) & = (w , \beta_{1,1}, \beta_{2,1}) (z , r_b , s_b)(xy^{-1} w(x^{-1},y^{-1}) y x^{-1}, |w|_x + \beta_{1,1} , |w|_y + \beta_{2,1}) \\
& = ( w z xy^{-1} w(x^{-1},y^{-1}) y x^{-1}, 2 \beta_{1,1} + |w|_x + r_b, 2 \beta_{2,1} + |w|_y + s_b) .
\end{align*}
By condition~(i), it follows that $|w|_x = - 2 \beta_{1,1}$ and $|w|_y = - 2 \beta_{2,1}$, and thus $w$ is non trivial. Furthermore, 
$w z x y^{-1} w(x^{-1},y^{-1}) y x^{-1}  = z$ by~(i), and so:
$$(y x^{-1} z^{-1}) w (y x^{-1} z^{-1})^{-1} = w(x^{-1},y^{-1})^{-1}.$$
By Lemma~\ref{palin_2}, we see that $|y x^{-1} z^{-1}|_x = -1 - |z|_x$ or $|y x^{-1} z^{-1}|_y = 1 - |z|_y$ is even, or equivalently, that $|z|_x$ or $|z|_y$ is odd. But this contradicts the fact that $|z|_x$ and $|z|_y$ are even. Thus $\beta$ has the Borsuk-Ulam property with respect to $\tau_2$ as required.
\end{proof}

\begin{proof}[Proof of Theorem \ref{BORSUK_TAU_2}]
The result follows directly from Propositions~\ref{classification_tau2_first_part} and~\ref{classification_tau2_second_part}.		
\end{proof}
		
\section{Proof of Theorem~\ref{BORSUK_TAU_3}}\label{borsuk_maps_klein}

The purpose of this section is to prove Theorem~\ref{BORSUK_TAU_3}. We identify $\pi_1 ( \klein )$ with $\zsdz$ as in Section~\ref{braid_klein}. Consider the following short exact sequence:
\begin{equation}\label{homo_tau_3}
\xymatrix{
1 \ar[r] & \pi_1 ( \klein ) \ar[r]^-{i_3} & \pi_1 ( \klein ) \ar[r]^-{\theta_3} & \ztwo \ar[r] & 1 	
}\end{equation}
		
\begin{center}
\begin{multicols}{2}
${\allowdisplaybreaks
i_3: \begin{cases}
(1,0) \longmapsto (2,0) \\
(0,1) \longmapsto (0,1)
\end{cases}}$		
		
${\allowdisplaybreaks
\theta_3: \begin{cases}
(1,0) \longmapsto \overline{1} \\
(0,1) \longmapsto \overline{0} .
\end{cases}}$
\end{multicols}
\end{center}
By covering space theory, there exists a double covering $c_3: \klein \to \klein$ such that the induced homomorphism on the fundamental group is $i_3$. Let $\tau_3: \klein \to \klein$ be the non-identity deck transformation associated with $c_3$. Then $\tau_3$ is a free involution, and~(\ref{homo_tau_3}) is the short exact sequence induced by $\tau_3$. Further, $\tau_{3}$ lifts to a homeomorphism $\widehat{\tau}_{3}: \rtwo \to \rtwo$, where $\widehat{\tau}_{3}(x,y)=(x,y+\frac{1}{2})$ for all $(x,y)\in \rtwo$. As the following result shows, up to the equivalence relation defined just after Proposition~\ref{reduce_involution}, this is the only free involution on the Klein bottle.

\begin{proposition}\label{class_involutions_klein} 
Let $\tau: \klein \to \klein$ be a free involution. Then $\tau$ is equivalent to $\tau_3$.
\end{proposition}
		
\begin{proof}
Since the natural projection $p_\tau: \klein \to \klein_\tau$ is a covering, $\klein_\tau$ is homeomorphic to the Klein bottle, and the  
short exact sequence induced by $\tau$ is of the following form:
\begin{equation}\label{seq_aux_3}\xymatrix{
 1 \ar[r] & \pi_1 ( \klein ) = \zsdz \ar[r]^-{(p_\tau)_\#} & \pi_1 ( \klein_\tau ) \ar[r]^-{\theta_\tau}& \ztwo \ar[r] & 1 .		
}\end{equation}
By Remark~\ref{identification_pi1_klein}  we have an isomorphism $\gamma: \zsdz \to \pi_1(\klein_\tau)$. The composition $\theta_{\tau}\circ \gamma$ 
is different from the homomorphism $\theta_2$ defined in (\ref{homo_tau_2}).  Indeed, if $\theta_{\tau}\circ \gamma=\theta_2$, then $\pi_1 ( \klein ) \cong \ker \theta_{\tau}\circ \gamma  = \ker \theta_2 \cong \pi_1 ( \torus )$, which is absurd. By~\cite[Proposition~32]{GonGua}, there exists an isomorphism $\psi: \pi_1 ( \klein ) \to \pi_1 ( \klein_{\tau} )$ such that $\theta_3 = \theta_\tau \circ \psi$. Comparing~(\ref{homo_tau_3}) and~(\ref{seq_aux_3}) and using exactness and the 5-Lemma, it follows that $\psi$ induces an isomorphism $\varphi: \pi_1 ( \klein ) \to \pi_1 ( \klein )$ such that $(p_\tau)_\# \circ \varphi = \psi \circ i_3$. So the exact sequences~(\ref{homo_tau_3}) and~(\ref{seq_aux_3}) are equivalent, and we conclude that $\tau$ and $\tau_1$ are equivalent by Proposition~\ref{invo_seq_exact}.
\end{proof}
		
By Theorem~\ref{set_homotopy}, the following diagram (we simplify the notation):
\begin{equation}\label{set_homo_klein}\begin{gathered}\xymatrix{
[ \klein , * ; \klein , * ]	\ar[r]^-{\Gamma}	\ar[d]_-{\Lambda} & \hom ( \zsdz , \zsdz ) \ar[d]^-{\Upsilon} \\
[ \klein , \klein ]  \ar[r]^-{\Delta} & \dfrac{ \hom ( \zsdz , \zsdz ) }{\sim}
}\end{gathered}\end{equation}
is commutative, where the maps $\Gamma$ and $\Delta$ are bijections and the maps $\Lambda$ and $\Upsilon$ are surjective. So, in order to determine the sets $[ \klein , * ; \klein , * ]$ and $[ \klein , \klein]$, we will describe the elements of the sets $\hom ( \zsdz , \zsdz )$ and $\dfrac{ \hom ( \zsdz , \zsdz ) }{\sim}$.

\begin{remark}\label{homomor_klein_1}
By~\cite[Lemma 3.1]{GonKel2}, for each $h \in \hom ( \zsdz , \zsdz )$, we have the following two possibilities for the images of the generators of $\zsdz$:
\begin{enumerate}[\textbullet]
\item $h(1,0) = (r_1 , 0)$ and $h (0,1) = ( r_2 , 2 s + 1 )$ for some $r_1, r_2 , s \in \z$. In this case, we say that $h$ is of \emph{type~A}.
\item $h(1,0) = (0 , 0)$ and $h (0,1) = ( r , 2 s )$ for some $r , s \in \z$. In this case, we say that $h$ is of \emph{type~B}.
\end{enumerate}
\end{remark}

\begin{proposition}\label{homomor_klein_2}
Let $h \in \hom ( \zsdz , \zsdz )$.
\begin{enumerate}[(a)]
\item If $h$ is of type A then there exist unique integers $i \in \{ 0 , 1 \}$, $r, s \in \z$, where $r \geq 0$, such that $h$ is conjugate to the homomorphism given by $\begin{cases}
(1,0) \longmapsto (r,0) \\
(0,1) \longmapsto (i,2s+1).
\end{cases}$
\item If $h$ is of type B then there exist unique integers $r, s \in \z$, where $r \geq 0$, such that $h$ is conjugate to the homomorphism given by $\begin{cases}
(1,0) \longmapsto (0,0) \\
(0,1) \longmapsto (r,2s).
\end{cases}$
\end{enumerate}
\end{proposition}

\begin{proof}
If $h: \zsdz \to \zsdz$ is of Type A, then conjugating $h$ by the element $(a,b) \in \zsdz$, we obtain a homomorphism $h^\prime : \zsdz \to \zsdz$ that depends on $r_{1}$ and $r_{2}$ as follows: 
\begin{enumerate}[\textbullet]
\item if $r_1 \geq 0$ (resp.\ $r_1 < 0$) and $r_2$ is even, let $a = - \frac{r_2}{2}$ and $b = 0$ (resp.\ $a =  \frac{r_2}{2}$ and $b = 1$). Then $h^\prime : \begin{cases}
(1,0) \longmapsto (\lvert r_1\rvert,0) \\
(0,1) \longmapsto (0,2s+1).
\end{cases}$

\item if $r_1 \geq 0$ (resp.\ $r_1 < 0$) and $r_2$ is odd, let $a =  \frac{- r_2 + 1}{2}$ and $b = 0$ (resp.\ $a =  \frac{r_2 + 1}{2}$ and $b = 1$). Then 
$h^\prime : \begin{cases}
(1,0) \longmapsto (\lvert r_1\rvert,0) \\
(0,1) \longmapsto (1,2s+1).
\end{cases}$
\end{enumerate}
If $h: \zsdz \to \zsdz$ is of Type B, then conjugating $h$ by $(0,1)$, we obtain a homomorphism $h^\prime: \zsdz \to \zsdz$ defined by: 
${\allowdisplaybreaks
h^\prime : \begin{cases}
(1,0) \longmapsto  (0, 0) \\
(0,1) \longmapsto (-r , 2s).
\end{cases}}$
We conclude that each element $h \in \hom ( \zsdz , \zsdz )$ is conjugate to one of the homomorphisms given in the statement of the proposition. Further, if $h_1, h_2 : \zsdz \to \zsdz$ are two such homomorphisms, then we leave it as an exercise to show that if they are distinct, then they are non conjugate, and the result follows. 
\end{proof}

\begin{remark}\label{set_homo_klein_2}
By diagram~(\ref{set_homo_klein}) and Proposition~\ref{homomor_klein_2}, there exists a bijection between the set $[ \klein , \klein ]$ and the subset of $\hom ( \zsdz , \zsdz )$ whose elements are described in parts~(a) and~(b) of Proposition~\ref{homomor_klein_2}. By abuse of notation, for each $\beta \in [ \klein , \klein ]$, the image of $\beta$ under the above bijection shall be denoted by $\beta_\#$.
\end{remark}

We now solve the Borsuk-Ulam problem for homotopy classes with respect to involution $\tau_3$. As in Lemmas~\ref{reduce_diagram_tau1} and~\ref{reduce_diagram_tau_2}, we have the following result:

\begin{lemma}\label{reduce_diag_tau3}
A free homotopy class $\beta \in [ \klein , \klein ]$ does not have the Borsuk-Ulam property with respect to  $\tau_3$ if and only if there exist pure braids $a,b \in P_2 ( \klein )$ such that:
\begin{enumerate}[(i)]
\item $l_\sigma ( a) l_\sigma ( b ) \sigma^2 a = b$.
\item $\beta_\# (1,0) = (p_1)_\# ( l_\sigma(a) a )$.
\item $\beta_\# (0,1) = (p_1)_\# ( b )$.
\end{enumerate}
\end{lemma}
		
\begin{proof}
First, suppose that $\beta$ does not have the Borsuk-Ulam property with respect to $\tau_3$. By~(\ref{set_homo_klein}) and Proposition~\ref{borsuk_pointed_free}, there exists a pointed homotopy class $\alpha$ that does not have the Borsuk-Ulam property such that $\alpha_\free = \Lambda ( \alpha ) =\beta$. By Proposition~\ref{borsuk_braid}, there exist homomorphisms $\varphi: \pi_1 ( \klein ) \to P_2 ( \klein )$ and $\psi : \pi_1 ( \klein ) \to B_2 ( \klein )$ that makes diagram~(\ref{diag_borsuk_braid}) commute. By (\ref{homo_tau_3}), $\psi(1,0)\in B_2 ( \klein ) \backslash P_2 ( \klein )$ and $\psi(0,1)\in P_2 ( \klein )$. From the short exact sequence~(\ref{seq_braid}) and Theorem~\ref{presentation_belin}, there exist $a,b \in P_2 ( \torus)$ such that: 
\begin{enumerate}[(1)]
\item $\psi(1,0) = \sigma a$.
\item $\psi(0,1) = b$.
\end{enumerate}
By Remark~\ref{identification_pi1_klein}, we have $\id = [ \psi(1,0) , \psi (0,1) ]^\prime = \sigma a b \sigma a b^{-1}$, which is equivalent to $(\sigma a \sigma^{-1}) ( \sigma b \sigma^{-1}) \sigma^2 a = b$, and which yields~(i). The commutativity of~(\ref{diag_borsuk_braid}) gives rise to the following relations:
\begin{enumerate}[(1)]\setcounter{enumi}{2}
\item $\varphi(1,0) = \psi ( (p_{\tau_1})_\# (1,0) ) = \psi(2,0) = (\sigma a )^2 = l_\sigma (a ) \sigma^2 a$.
\item $\varphi(0,1) = \psi ( (p_{\tau_1})_\# (0,1) ) = \psi(0,1) = b$.
\end{enumerate}
Finally, 
$\beta_\#  (1,0) 
= (p_1)_\# ( \varphi(1,0)) 
= (p_1)_\# ( l_\sigma (a ) \sigma^2 a ) 
= (p_1)_\# ( l_\sigma ( a) a)$ and 
$ \beta_\# (0,1) = (p_1)_\# ( \varphi(0,1)) = (p_1)_\# ( b )$ by~(\ref{set_homo_klein}), from which we obtain~(ii) and~(iii). Conversely, suppose that there exist $a, b \in P_2 ( \klein )$ such that conditions~(i),~(ii) and~(iii) are satisfied. We define $\psi : \pi_1 ( \klein ) \to P_2 ( \klein )$ (resp.\ $\varphi: \pi_1 ( \klein ) \to B_2 ( \klein )$) on the generators of $\pi_1 ( \klein )$ by equations~(1) and~(2) (resp.~(3) and~(4)). These maps extend to homomorphisms because:
\begin{align*}
[ \varphi (1,0) , \varphi(0,1) ]^\prime  &=  l_\sigma( a ) \sigma^2 a  b l_\sigma( a ) \sigma^2 a  b^{-1}
=  l_\sigma ( a l_\sigma (a) l_\sigma (b)  \sigma^2 a ) \sigma^2 a b^{-1} \stackrel{\text{(i)}}{=}   l_\sigma( a b ) \sigma^2 a b^{-1} \\
& =  l_\sigma ( a) l_\sigma(b) \sigma^2 a b^{-1}
\stackrel{\text{(i)}}{=} \id\;\text{and}\\
[ \psi(1,0) , \psi(0,1) ]^\prime
&=  \sigma a b \sigma a  b^{-1}
=  l_\sigma(a) l_\sigma(b) \sigma^2 a b^{-1}
\stackrel{\text{(i)}}{=}  \id.		
\end{align*}
By a calculation similar to that given in the first part of the proof, we see that diagram~(\ref{diag_borsuk_braid}) is commutative. So by Proposition~\ref{borsuk_braid} and~(\ref{set_homo_klein}), $\beta$ does not have the Borsuk-Ulam property with respect to $\tau_3$.
\end{proof}
		
\begin{proposition}\label{classification_tau3_first_part}
With the notation of Proposition~\ref{homomor_klein_2} and Remark~\ref{set_homo_klein_2}, let $\beta \in [ \klein , \klein ]$ be such that the homomorphism $\beta_\# : \zsdz \to \zsdz$ is of Type~A. Then $\beta$ does not have the Borsuk-Ulam property with respect to $\tau_3$.
\end{proposition}
		
\begin{proof}
First, suppose that $r$ is even. Let $m\in \mathbb{Z}$ be such that $r = 2m$. With the notation of Theorem~\ref{braid_group_klein}, consider the elements $a = ( \id ; m , 0)$ and $b = ( B ; i , 2s+1 )$ of $P_2 ( \klein )$. Let us show that these elements satisfy the hypotheses~(i),~(ii) and~(iii) of Lemma~\ref{reduce_diag_tau3}, from which it will follow that $\beta$ does not have the Borsuk-Ulam property with respect to $\tau_3$. Using Theorem~\ref{action_sigma_k2}, we have:
\begin{align*}
l_\sigma (a) l_\sigma (b) \sigma^2 a & = l_\sigma ( \id ; m , 0 ) l_\sigma ( B ; i , 2s +1) (B ; 0 , 0) ( \id ; m , 0 ) \\
& = ( \id ; m , 0 ) l_\sigma( B ; 0 , 0) l_\sigma ( \id ; i , 0 ) l_\sigma(\id ; 0 , 2s + 1) (B ; m , 0) \\
& =  ( \id ; m , 0 ) ( B ; 0 , 0) ( \id ; i , 0 )(B ; 0 , 2s + 1) (B ; m , 0) \\
& = ( \theta(m,0)(B) ; m + i , 0 )( B \theta(0,2s+1)(B) ; -m , 2 s + 1) \\
& = ( B ; m + i , 0) ( \id ; -m , 2s + 1) = ( B ; i , 2 s + 1 ) = b,
\end{align*}
and hence condition~(i) is satisfied. Further,
\begin{align*}
(p_1)_\# ( l_\sigma(a) a)  &= (p_1)_\# ( l_\sigma ( \id ; m , 0 ) ( \id ; m , 0 ))
= (p_1)_\# (  ( \id ; m , 0 ) ( \id ; m , 0 ))
= (2m,0) \\
& = (r,0) = \beta_\# (1,0),\; \text{and}\\
(p_1)_\# ( b )  &= (p_1)_\# ( B ; i , 2s +1) = (i,2s + 1) = \beta_\# (0,1),
\end{align*}
and thus conditions~(ii) and~(iii) are also satisfied.

Now suppose that $r$ is odd, and let $m\in \mathbb{Z}$ be such that $r = 2 m + 1$. Consider the elements $a = ( u B^{-m} ; m , 0)$ and $b = ( u^{-1} B^{\delta_{(i+1)}} ; i , 2s + 1)$ of $P_2 ( \klein)$. The conclusion follows in a manner similar to that of the previous case, and the details are left as an exercise.
\end{proof}

\begin{proposition}\label{classification_tau3_second_part}
With the notation of Proposition~\ref{homomor_klein_2} and Remark~\ref{set_homo_klein_2}, let $\beta \in [ \klein , \klein ]$ be such that the homomorphism $\beta_\# : \zsdz \to \zsdz$ is of Type~B. Then $\beta$ has the Borsuk-Ulam property with respect to $\tau_3$.		
\end{proposition}

\begin{proof}
We argue by contradiction. Suppose that $\beta$ does not have the Borsuk-Ulam property with respect to $\tau_3$. By Proposition~\ref{borsuk_pointed_free} and diagram~(\ref{set_homo_klein}), there exists a pointed map $f: ( \klein , * ) \to ( \klein , *)$ such that $f_\# = \beta_\#$ and $f( \tau_3 (x)) \neq f(x)$ for all $x \in \klein$. Notice that the image of $f_\#$ is contained in the image of the homomorphism $i_2 : \pi_1 ( \torus ) \to \pi_1 ( \klein )$ that is defined in~(\ref{homo_tau_2}). Recall that $c_2 : \torus \to \klein$ is the double covering such that $(c_2)_\# = i_2$. So, by covering space theory, there exists a pointed map $\widetilde{f}: ( \klein , * ) \to ( \torus , * )$ such that $c_2 \circ \tilde{f} = f$. Therefore, the map $\tilde{f}$ satisfies $\widetilde{f}( \tau_3 ( x)) \neq \widetilde{f}( x)$ for all $x \in \klein$, which implies that the triple $( \klein, \tau_3 ; \torus)$ does not have the Borsuk-Ulam property. But this yields a contradiction using Remark~\ref{identification_pi1_klein}, the short exact sequence~(\ref{homo_tau_3}) and~\cite[Proposition 10]{GonGua}. The result then follows.
\end{proof}
		
We are now able to prove the final main result of this paper.

\begin{proof}[Proof of Theorem~\ref{BORSUK_TAU_3}]
Using arguments similar to those of Proposition~\ref{classification_tau3_second_part} and the notation of Remark~\ref{set_homo_klein_2}, a homotopy class $\beta \in [ \klein , \klein ]$ lifts to the torus if and only if the homomorphism $\beta_\# : \zsdz \to \zsdz$ is of Type~B. Further, if $H: \klein \to \klein$ is a homeomorphism and there exist maps $f_1 , f_2 : \klein \to \klein$ such that $f_2 = f_1 \circ H^{-1}$, then $f_1$ lifts to the torus if and only if $f_2$ lifts to the torus. The result then follows easily from Propositions~\ref{reduce_involution},~\ref{class_involutions_klein},~\ref{classification_tau3_first_part} and~\ref{classification_tau3_second_part}.
\end{proof}
		
\subsection*{Acknowledgements}

Part of this work is contained in  the Ph.D. thesis~\cite{Laass} of the third author who was supported by CNPq project nº~140836 and Capes/COFECUB project nº~12693/13-8. The first and second authors were partially supported by CNRS/FAPESP project number (2014/50131-7: Brazil; 2014/226555: France). The second author wishes to thank the `R\'eseau Franco-Br\'esilien en Math\'ematiques' for financial support during his visit to the Instituto de Matem\'atica e Estat\'istica, Universidade de S\~ao Paulo, from the 9\textsuperscript{th} of July to the 1\textsuperscript{st} of August 2016.


\begin{thebibliography}{30}

{\small

\bibitem{Bel} P.~Bellingeri, {\it On presentations of surface braid groups}, J.~Algebra \textbf{274} (2004), 543--563.

\bibitem{BelGerGua} P.~Bellingeri, S.~Gervais, J.~Guaschi, {\it Lower central series of Artin–Tits and surface braid groups}, J.~Algebra \textbf{319} (2008), 1409--1427. 

\bibitem{Bir} J.~S.~Birman, {\it Braids, Links, and Mapping Class Groups}, Ann.\ Math.\ Stud.~\textbf{82}, Princeton University Press (1974).

\bibitem{Borsuk} K.~Borsuk, {\it Drei Sätze über die $n$-dimensionale Euklidische Sphäre}, Fund.\ Math.\ \textbf{20} (1933), 177--190.

\bibitem{FadHus} E.~Fadell, S.~Husseini, {\it The Nielsen number on surfaces}, Contemp.\ Math.\ \textbf{21} (1983), 59--99.

\bibitem{FadellNeu} E.~Fadell, L.~Neuwirth, {\it Configuration Spaces}, Math.\ Scand.\ \textbf{10} (1962), 111--118.

\bibitem{Gon} D.~L.~Gon\c{c}alves, {\it The Borsuk-Ulam theorem for surfaces}, Quaest.\ Math.\ \textbf{29} (2006), 117--123.

\bibitem{GonGua} D.~L.~Gon\c{c}alves, J.~Guaschi, {\it The Borsuk-Ulam theorem for maps into a surface}, Top.\ Appl.\ \textbf{157} (2010), 1742--1759.

\bibitem{GonHaZe}  D.~L.~Gon\c{c}alves, C.~Hayat,  P.~Zvengrowski,  {\it The Borsuk-Ulam theorem for manifolds, with applications to dimensions two and three}, Proceedings of the International Conference, Bratislava Topology Symposium,  Group Actions and Homogeneous Spaces \rm, (2010), 1--12.

\bibitem{GonKel2} D.~L.~Gon\c{c}alves, M.~R.~Kelly, {\it Wecken type problems for self-maps of the Klein bottle}, Fixed Point Theory Appl.\ (2006), 1--15.

\bibitem{Laass} V.~C.~Laass, {\it A propriedade de Borsuk-Ulam para funções entre superfícies}, Ph.D Thesis, IME, Universidade de São Paulo (2015).

\bibitem{Magnus} W.~Magnus, A.~Karrass, D.~Solitar, {\it Combinatorial Group Theory: Presentations of groups in terms of generators and relations}, 2nd edition, Dover Publications, Inc., Mineola, NY, (2004).

\bibitem{Mato} J.~Matou\v{s}ek, {\it Using the Borsuk-Ulam Theorem}, Universitext, Springer--Verlag (2002).

\bibitem{Mura} K.~Murasugi, B.~I.~Kurpita, {\it A Study of Braids}, Mathematics and Its Applications \textbf{484}, Kluwer Academic Publishers, (1999).

\bibitem{Nielsen} J.~Nielsen, {\it Untersuchungen zur Topologie der geschlossenen zweiseitigen Flächen}, Acta Math.\ \textbf{50} (1927), 189--358.

\bibitem{Vick} J.~W.~Vick, {\it Homology Theory: An Introduction to Algebraic Topology}, Graduate Texts in Mathematics \textbf{145}, Springer--Verlag (1994).

\bibitem{White} G.~W.~Whitehead, {\it Elements of homotopy theory}, Graduate Texts in Mathematics \textbf{61}, Springer--Verlag (1978).

\bibitem{Zies} H.~Zieschang, E.~Vogt, H.-D.~Coldewey, {\it Surfaces and Planar Discontinuous Groups}, Lecture Notes in Mathematics \textbf{835}, Springer--Verlag (1980).

}

\end{thebibliography}
\end{document}